\documentclass[11pt, final, a4paper, reqno]{article}
\usepackage[utf8]{inputenc}
\usepackage{hyperref}
\usepackage{titlesec}
\usepackage{braket}
\usepackage{mathtools}
\usepackage[T1]{fontenc}
\usepackage{amsmath,amssymb,amsfonts,amsthm,shuffle,xargs}
\usepackage{dsfont}
\usepackage{tikz}
\usetikzlibrary{cd}
\usepackage{comment}
\usepackage{envmath}
\usepackage[a4paper,vmargin={3.5cm,3.5cm},hmargin={2.5cm,2.5cm}]{geometry}
\tikzstyle{guillpart} = [scale=0.5,baseline={(current bounding box.center)}]
\tikzstyle{guillfill}=[gray,opacity=0.5]
\tikzstyle{guillsep} = [ultra thick]
\tikzstyle{centerline} = [baseline={(current bounding box.center)}]
\tikzstyle{line}=[linewidth= 0.2pt,baseline={(current bounding box.center)}]

\newcommand{\setC}{\mathbb{C}}
\newcommand{\setN}{\mathbb{N}}
\newcommand{\setZ}{\mathbb{Z}}
\newcommand{\setR}{\mathbb{R}}

\newcommand{\EE}{\mathbb{E}}
\newcommand{\ii}{\textbf{i}}

\newcommand{\tr}{\text{Tr}}

\newcommand{\ee}{\boldsymbol{e}}

\newcommand{\vect}[2]{\begin{pmatrix}
		#1\\
		#2
\end{pmatrix}}

\newcommand{\covg}[2]{\mathrm{Cov}_{\mathrm{Gibbs}}\left(#1,#2\right)}

\newcommand{\covcg}[3]{\mathrm{Cov}_{\mathrm{Gibbs}}\left(#1,#2 \vert #3\right)}

\newcommand{\var}[1]{\mathrm{Var}\left(#1\right)}

\theoremstyle{plain}
\newtheorem{theorem}{Theorem}[section]
\newtheorem{property}{Property}[section]
\newtheorem{lemma}{Lemma}[section]

\newtheorem{corollary}{Corollary}[section]

\theoremstyle{definition}
\newtheorem{definition}{Definition}[section]
\newtheorem{rmk}{Remark}[section]

\newtheorem{assumptionalt}{Assumption}

\newenvironment{assumptionp}[1]{
	\renewcommand\theassumptionalt{#1}
	\assumptionalt
}{\endassumptionalt}

\tikzstyle{guillpart} = [scale=0.5,baseline={(current bounding box.center)}]
\tikzstyle{guillfill}=[gray,opacity=0.5]
\tikzstyle{guillsep} = [ultra thick]
\tikzstyle{guilldot} = [dotted, ultra thick]
\tikzstyle{centerline} = [baseline={(current bounding box.center)}]

\title{One-dimensional discrete Gaussian Markov processes: Harmonic decomposition of invariant boundary conditions}
\author{Emilien Bodiot}

\begin{document}
	\maketitle
\begin{abstract}
	We study invariant boundary conditions for one dimensional discrete Gaussian Markov processes, basic toy models of spatial Markov processes in statistical mechanics. More precisely, we give a decomposition of boundary objects in a non trivial basis from the study of a meromorphic matrix-valued function $\Phi$ (inherent to the model) and its singularities. This provides a simple algorithm for the explicit computation of invariant measures. As an application, we give an "eigen" version of Szegő limit theorem for matrix valued trigonometric polynomials.
\end{abstract}
\tableofcontents

	\section{Introduction and results}
Gaussian Markov fields have been widely studied and arise in numerous research areas such as quantum field theory, statistical physics, statistics... For instance, continuous Gaussian Free Field appears in Liouville quantum gravity, see \cite{berestycki2021gaussian}, harmonic crystal in physics, and discrete Gaussian Markov fields in statistics, see \cite{rue2005gaussian}. Often, boundary conditions are taken to be periodic or deterministic (like zero boundary conditions). Rarely, invariant boundary conditions are studied.

In this article we will focus on the study of invariant boundary conditions for Gaussian Markov Processes on the lattice $\setZ$. Of course, it is well-known that most of interesting physical quantities, such as free energy and correlation functions, are easily and explicitly computable using Fourier transform when working with periodic boundary conditions (due to invariance by translation). Nevertheless, these global computations totally hide the Markov property and the associated rich structure from which the model inherits. Another method is to take advantage of the Markov property and look for invariant measures. However, their explicit expression requires to solve tedious non-linear equations, which is hard in practice. Nevertheless, this local approach also leads to the computation of the same interesting physical quantities. It is remarkable since we do not know any link between invariant measures and Fourier transform. 

The main objective of this article is to link these approaches together and, more importantly, to provide alternative computation tools and simple algorithms for the explicit calculation of invariant boundaries. To do so, we will investigate algebraic and recursive properties of the model and highlight its links with variety of other fields such as residue calculus, discrete harmonic analysis, Toeplitz analysis, etc... This will lead us to a decomposition of boundary objects in a non trivial basis from the study of a meromorphic matrix-valued function $\Phi$ (inherent to the model) and its singularities. This article may also be viewed as a warm up for forthcoming work on the higher dimensional case. 

After introducing and discussing the model, we expose the main results of this article. 
\subsection{Definitions and classical facts}
This section is devoted to define the model and expose classical computations and results. 
\subsubsection{Definition of the model}
\paragraph{The lattice $\setZ$.}In the whole article, $\setZ$ will be considered as a lattice composed of vertices $V$ and unoriented edges $E$,
$$V=\setZ$$
$$E=\lbrace \set{x,y}\subset V; \vert x-y\vert=1 \rbrace.$$
For any $e\in E$, we denote $L(e)=\min e$ the left vertex of $e$ and $R(e)=\max e$ the right vertex of $e$.
We say that $D$ is a (finite) domain  of $\setZ$ if $D$ is a (finite) subset of $E$. We denote its associated vertices $V(D)$, formally,
$$V(D)=\lbrace v\in V; v=L(e) \text{ or } v=R(e) \text{ for } e\in D\rbrace.$$
Moreover, we write $\partial D$ the set of its boundary vertices defined as the vertices which belong exactly to one edge in $D$, that is
$$\partial D=\lbrace v \in V(D); \exists! y\in V(D) \text{ s.t }  \set{v,y}\in E(D) \rbrace.$$ 
A partition of $D$ is a finite collection of non-empty domains $(D_i)_{i\in I}$ such that $D_i\cap D_j=\emptyset$ for $i\neq j$ and $\cup_{i\in I} D_i=D$. 

We say that $D$ is connected if $D$ is only composed of consecutive edges, formally $D$ is of form
$$D=\lbrace \lbrace L_D+k,L_D+k+1\rbrace \in E; k\in \set{0,\dots, R_D-L_D-1}\rbrace$$
for $L_D<R_D$ left and right vertices of $D$. Up to relabeling edges of $D$, we identify $D$ with
$$D=\lbrace \lbrace k,k+1\rbrace \in E; k\in \set{0,\dots, l(D)-1}\rbrace$$
and  $l(D)=R_D-L_D$ is the size of $D$.
\paragraph{Markov property on $\setZ$.}
We will only consider non-oriented Markov processes as we are, due to physical motivations, more interested in spatial processes rather than processes indexed by time. We recall the definition of the spatial Markov property.  
\begin{definition}[Spatial Markov]\label{markov}
	Let $(G,\mathcal{G})$ be a measurable space, $D$ a finite domain of $\setZ$ and $X^{(D)}=(X_v^{(D)})_{v\in  V(D)}$  a set of $(G,\mathcal{G})$-valued random variables on a probability space $(\Omega,\mathcal{F},\mathbb{P})$. We say that $X^{(D)}$ has the Markov property on $D$ if and only if for all partition $(D_i)_{i\in I}$ of $D$ and any bounded measurable function $h_i:G^{\vert D_i \vert}\to\setC$, we have
	$$\EE\left[\prod_{i\in I}h_i(X^{(D_i)})\middle\vert \sigma(X_v;v\in\cup_{i\in I}\partial D_i)\right]=\prod_{i\in I}\EE[h_i(X^{(D_i)})\vert \sigma(X_v;v\in \partial D_i)].$$
\end{definition}
\begin{rmk}
	If $D$ is not connected we can always write $D$ as a disjoint union of its connected components which induces a rewriting of $\partial D$ as a disjoint union of boundaries of its connected components. Then, $X^{(D)}$ can be decomposed in independent Markov processes defined on connected components of $D$. Therefore, without loss of generality, we restrict ourselves to the case of connected 
	domains $D=\lbrace \lbrace k,k+1\rbrace; k\in \set{0,\dots, P-1}\rbrace$ for the study Markov processes on $\setZ$. 
\end{rmk}
\paragraph{Gaussian weights and Gaussian boundary weights.}For all $d\in\setN^*$, $\alpha\in\setR_+^*$, $A\in M_{2d}(\setC)$ and $x$, $y\in\setC^d$, we denote
$$\ee_{\alpha,A}(x,y)\coloneqq\alpha \exp\left(-\frac{1}{2}\begin{pmatrix}
	x\\
	y
\end{pmatrix}^*A\begin{pmatrix}
	x\\
	y
\end{pmatrix}\right)$$
and for all $\beta\in\setR_+^*$ and $B\in M_d(\setC)$,
$$\ee_{\beta,B}'(x)=\beta\exp\left(-\frac{1}{2}x^*Bx\right),$$
(cf. equations \eqref{gweight} and \eqref{gbweight} for details). Constants $\alpha$ and $\beta$ do not play an important role in the model but we will keep them for potential renormalizations (see theorem \ref{propre} for instance).   Moreover, we will often refer to the block decomposition of $A\in M_{2d}(\setC)$,
\begin{equation}\label{blockdecompo}
	A=\begin{pmatrix}
		A_{LL} & A_{LR}\\
		A_{RL} & A_{RR}
	\end{pmatrix}
\end{equation}
with $A_{IJ}$ $(d\times d)$-square matrices  for all $I$ $J\in\set{L,R}$.

\paragraph{Gaussian Markov processes on $\setZ$.}
\begin{definition}[$1D$ Gaussian Markov Processes]\label{Gauss_Markov_Proc}
	Let $(\Omega, \mathcal{F},\mathbb{P})$ be a probability space. Let $d\in\setN^*$ and $D$ a connected domain of $\setZ$ of size $l(D)=P\in\setN^*$. A $\setC^d$-valued Gaussian Markov process on $D$  of weights $(\ee_{\alpha^{(k)},A^{(k)}})_{k=0}^{P-1}$ with boundary conditions $(\ee_{\beta_{L},B_{L}}',\ee_{\beta_{R},B_{R}}')$ is a collection of $\setC^d$-valued random variables $X^{(D)}=(X_k^{(D)})_{k=0}^{P}\in\setC^{d(P+1)}$, indexed by vertices of $D$, with density (with respect to product Lebesgue measure) given by
	\begin{align*}
		&g_{X^{(D)}}(x^{(D)})\\
		&=\frac{1}{Z_D(\alpha^{(\bullet)} ,A^{(\bullet)},\beta_L,B_L,\beta_R,B_R)} \ee_{\beta_{L},B_{L}}'(x_{0})\left(\prod_{k=0}^{P-1}\ee_{\alpha^{(k)},A^{(k)}}(x_{k},x_{k+1})\right)\ee_{\beta_{R},B_{R}}'(x_{P})\\
		&=\frac{\beta_L\beta_R\prod_{k=0}^{P-1}\alpha^{(k)}}{Z_D(\alpha^{(\bullet)} ,A^{(\bullet)},\beta_L,B_L,\beta_R,B_R)}\exp\left(-\frac{1}{2}(x^{(P)})^* Q^{(D)}_{(A^{(\bullet)}, B_L,B_R)}x^{(P)}\right)
	\end{align*}
	for all $x^{(D)}=(x_k^{(D)})_{k=0}^P\in\setC^{d(P+1)}$ with $\alpha^{(k)}\in\setR_+^*$ for all $k$, $\beta_L$, $\beta_R \in\setR_+^*$ and the matrices $(A^{(k)})_{k=0}^{P-1}$, $B_L$ and $B_R$ are chosen so that
	\begin{align}
		&Q^{(D)}_{(A^{(\bullet)}, B_L,B_R)}\nonumber\\
		&= \begin{pmatrix}
			B_L +A^{(0)}_{LL} & A^{(0)}_{LR} & 0 & 0 & \dots & 0\\
			A^{(0)}_{RL} & A^{(0)}_{RR}+A^{(1)}_{LL} & A^{(1)}_{LR} & 0 & \dots & 0\\
			0 & A^{(1)}_{RL} & A^{(1)}_{RR}+A^{(2)}_{LL} & A^{(2)}_{LR} & \ddots & 0\\
			\vdots & \ddots & \ddots & \ddots & \ddots & \vdots\\
			0 & \dots & 0 & A_{RL}^{(P-2)} & A^{(P-2)}_{RR}+A^{(P-1)}_{LL} & A^{(P-1)}_{LR}\\
			0 & 0 & \dots & 0 &A^{(P-1)}_{RL} & A^{(P-1)}_{RR}+B_R
		\end{pmatrix}\label{matrice_couplage}
	\end{align}
	is Hermitian and positive definite. 
	Moreover, the normalizing constant $Z_D(\alpha^{(\bullet)} ,A^{(\bullet)},\beta_L,B_L,\beta_R,B_R )$, the so-called partition function in statistical mechanics, is given by
	\begin{align}
		Z_D(\alpha^{(\bullet)} ,A^{(\bullet)},\beta_L,B_L,\beta_R,B_R )&=\int_{\setC^{d(P+1)}}\ee_{\beta_L,B_L}'(x_{0})\left(\prod_{k=0}^{P-1}\ee_{\alpha^{(k)},A^{(k)}}(x_{k},x_{k+1})\right)\ee_{\beta_R,B_R}'(x_{P})dx^{(D)}\\
		&=(2\pi)^{d(P+1)}\det \left(Q^{(D)}_{(A_\bullet , B_L,B_R)}\right)^{-1}\beta_L\beta_R\prod_{k=0}^{P-1}\alpha^{(k)}.\nonumber
	\end{align}
	
	We say that this process is homogeneous of weight $(\alpha,A)$ if $\alpha^{(k)}=\alpha$ and $A^{(k)}=A$ for all $k\in \lbrace 0,\dots,P-1\rbrace$.  
\end{definition}
It is clear, from the factorized form, that any Gaussian Markov Process has indeed the Markov property. This justifies, as long as we stick to finite dimension, defining Gaussian processes by their density rather than by their covariance matrices on which the Markov property is less easy to read.
\begin{rmk}
	Our processes are $\setC^d$-valued circularly symmetric ($e^{\ii\phi}X\sim X$ for any $\phi\in[-\pi,\pi[$) only to simplify discrete Fourier transform computations. We could have chosen $\setR^d$-valued processes and most of results remain the same (up to square roots on determinants). 
\end{rmk}
\begin{rmk}
	In the definition of Gaussian Markov Processes \ref{Gauss_Markov_Proc}, the only assumption on the matrices involved is $Q^{(D)}_{(A^{(\bullet)} , B_L,B_R)}$ being Hermitian and positive definite. This ensures that the random vector is indeed a multivariate Gaussian on $\setC^{d(P+1)}$. Of course, this global assumption implies local properties such as
	$$\begin{pmatrix}
		A^{(k-1)}_{RR} +A^{(k)}_{LL} & A^{(k)}_{LR}\\
		A^{(k)}_{RL} & A^{(k)}_{RR} +A^{(k+1)}_{LL}
	\end{pmatrix}$$
	being Hermitian and positive definite for $k\in\lbrace 1,\dots k-2\rbrace$. Nevertheless, this is not a property on $A^{(k)}$ itself but on $A^{(k)}$ surrounded by its neighbors $A^{(k-1)}$ and $A^{(k+1)}$. 
	In a view of giving a consistent algebraic structure to Gaussian weights, we are interested in considering $A^{(k)}$ (thus $\ee_{\alpha^{(k)},A^{(k)}}$) on its own. What property the matrices $(A^{(k)})_k$ should satisfy? Same question holds for $B_L$ and $B_R$. The assumption that $(A^{(k)})_k$, $B_L$ and $B_R$ are Hermitian and positive definite basically allows to do whatever we want in the "Gaussian world": marginalization, computing observables etc... Moreover, the (homogeneous) model admits a unique infinite volume Gibbs measure if the inverse of a certain function $\Phi$, defined in section \ref{phi}, is not singular on the unit circle, see section 13.3 of \cite{georgii2011gibbs}. This is automatic when $(A^{(k)})_k$, $B_L$ and $B_R$ are assumed to verify Hermitian and positive definite properties. We will therefore work with the following set.
\end{rmk} 

\paragraph{The set $H_d^+(\setC)$.}
For any $d\in\setN^*$, we introduce the following set 
\begin{equation}\label{herm_pos_def}
	H_d^+(\setC)=\lbrace M\in M_d(\setC); M\text{ is Hermitian positive definite}\rbrace.
\end{equation}
In this article, due to the latter discussion, we will only be considering matrices $A^{(k)}\in H_{2d}^+(\setC)$, and $B_L$, $B_R \in H_{d}^+(\setC)$.

\paragraph{Particular cases.} To highlight the previous assumption, here are some specific cases covered or not by the model.
\begin{itemize}
	\item Discretised Orstein-Uhlenbeck is a special case of the model we are considering. Let $d=1$ and take $m>0$. The Gaussian Markov process with weight $\ee_{\alpha,A}$ for some $\alpha\in\setR_+^*$ and
	$$A=\begin{pmatrix}
		1+\frac{m^2}{2} & -1\\
		-1 & 1 + \frac{m^2}{2}
	\end{pmatrix}\in H_2^+(\setC).$$
	is a discretised Orstein-Uhlenbeck. It corresponds to the Hamiltonian
	$$H(x)=\frac{1}{2}\sum_k(x_{k+1}-x_k)^2+\frac{m^2}{2}\sum_k x_k^2$$
	of a one-dimensional massive Gaussian free field of mass $m$, see \cite{friedli2017statistical} for instance.
	\item Discretised Brownian motion (or one-dimensional harmonic crystal) does not verify the positive definite assumption and is therefore not covered by the model. Indeed, discretised Brownian motion requires to work with following kernel
	$$\ee(x_k,x_{k+1}) \sim e^{-\frac{1}{2}\vert x_{k+1} - x_k \vert^2}=\exp\left({-\frac{1}{2}\vect{x_k}{x_{k+1}}^*\begin{pmatrix} 1 & -1\\
			-1 & 1
		\end{pmatrix}\vect{x_k}{x_{k+1}}}\right)$$
	and $\begin{pmatrix} 1 & -1\\
		-1 & 1
	\end{pmatrix}\notin H_2^+(\setC)$ since it is not full rank. However, not covering the Brownian case is not a problem as we are interested, in this article, in invariant probability measures (which discretised Brownian motion does not have).
\end{itemize}
\paragraph{Homogeneous case.} Except if the contrary is mentioned, we will stick to the homogeneous case. Let $l(D)=P\in\setN^*$ be the size of $D$. In that case, $X^{(D)}$, the partition function, and the matrix $Q^{(D)}_{(A_\bullet , B_L,B_R)}$ are simply denoted, respectively, $X^{(P)}$, $Z_P(\alpha,A,\beta_L,B_L,\beta_R,B_R)$ and $Q^{(P)}_{(A,B_L,B_R)}$. We say that $X^{(P)}$ is a (homogeneous) Gaussian Markov process of size $P$. 
\paragraph{The full rank assumption.} Remark that if $A_{LR}$ is not full rank, so is $A_{RL}$ and $x$ and $y\in\setC^d$ are only coupled by $\ee_{\alpha,A}(x,y)$ through $x_{\ker(A_{RL})^\perp}$ and $y_{\ker(A_{LR})^\perp}$, their orthogonal projection onto, respectively, $\ker(A_{RL})^\perp$ and $\ker(A_{LR})^\perp$. Therefore, in most of the results we adopt the following assumption.
\begin{assumptionp}{I}\label{full_rank} The matrix $A$ satisfies
	\begin{equation}
		\mathrm{rank}(A_{LR})=d.
	\end{equation}
\end{assumptionp}

Section \ref{non-invertible} is devoted to the extension of these results in the case where $A_{LR}$ is not full rank. This only makes computations much heavier and does not change the nature of the results.
\subsubsection{Classical motivations and the problem with boundaries}
Given a Gaussian Markov process of weight $\ee_{\alpha,A}$ with boundary conditions $(\ee_{\beta_L,B_L}',\ee_{\beta_R,B_R}')$, we could be interested in computing the following limit
\begin{equation}
	f(\alpha,A,\beta_L,B_L,\beta_R,B_R)=\lim_{P \to \infty}\frac{1}{P}\log Z_P(\alpha,A,\beta_L,B_L,\beta_R,B_R),
\end{equation}
the so-called free energy. Another quantity of interest is the correlation function in the limit $P \to \infty$. That is, for $n\leq P$ and for any $v_1,\dots, v_n\in V(D)$, the computation of
\begin{equation}
	\lim_{P\to\infty}\mathbb{E}_{B_L,B_R}[X_{v_1}^{(P)}\cdots X_{v_{n}}^{(P)}],
\end{equation}
which, for Gaussian processes reduces to the limit of the covariance matrix by Wick theorem.  Being able to define the model on the whole lattice $\setZ$ is also one of the main motivation when dealing with statistical physics problems. This point is particularly easy here since for an infinite dimensional Gaussian measure to exist, one only need to give its covariance operator. Therefore, the computation of the correlation function will lead to the existence of the associated infinite volume Gibbs measure.
\paragraph{The choice of boundary conditions.} Remark that, no matter which quantity we wish to compute, we face the problem of specifying the boundary conditions $B_L$ and $B_R$. As explained bellow, some choices are particularly efficient. Nevertheless, because we expect exponential decorrelation (see remark \ref{decrease}), these quantities should not be depending on the initial boundary settings we choose. The present paper discusses the similarities between two common boundary conditions: periodic and invariant boundary conditions. 
\subsubsection{Various tools of study}
Let us briefly present various ways to compute those quantities. 
\paragraph{Via Discrete Fourier Transform.} One common approach is to work with Gaussian Markov Processes on the discrete torus $\setZ/P\setZ$. Then, with the help of discrete Fourier transform, computations become easier.
\paragraph{Notation.}
Let $P\in\setN^*$, and $k\in\setN$. We write
\begin{equation}
	\omega_{P}^{k}\coloneqq e^{\frac{2\ii\pi k}{P}}.
\end{equation}
Under periodic boundary conditions, we have the following theorem which answers most of classical questions. 
\begin{theorem}\label{Fourier}
	Let $P\in\setN^*$, $\alpha\in\setR_+^*$, and $A\in H_{2d}^+(\setC)$. Let $X^{(P),\mathrm{per}}=(X_k^{(P),\mathrm{per}})_{k\in\setZ/P\setZ}$ be a $\setC^d$-valued Gaussian Markov process on $\setZ/P\setZ$ with weight $\ee_{\alpha,A}$. Denote for all $k\in \setZ/P\setZ$, the Fourier transform
	$$\widehat{X}_k^{(P)}=\frac{1}{\sqrt{P}}\sum_{i\in\setZ/P\setZ}X_i^{(P),\mathrm{per}}\omega_{P}^{ik}.$$
	Then, the r.v $(\widehat{X}_k^{(P)})_{k\in\setZ/P\setZ}$ are independent and each $\widehat{X}_k^{(P)}$ has density against Lebesgue measure
	$$\frac{1}{2\pi \det\Phi_A(\omega_P^k)^{-1}}\exp\left(-\frac{1}{2}x^*\Phi_{A}(\omega_{P}^{k})x\right)$$
	where $\Phi_A:\setC^*\to M_d(\setC)$ is a matrix-valued function built from blocks of $A$, non singular on the unit circle $S^1$, see \eqref{defphi}. We then have
	$$Z_P^{\mathrm{per}}(\alpha,A)=\alpha^P(2\pi)^{dP}\prod_{k=0}^{P-1}\det \Phi_{A}(\omega_P^k)^{-1}$$
	so that the free energy $f^{\mathrm{per}}(\alpha,A)$ is equal to
	\begin{equation}\label{free_energy_per}
		f^{\mathrm{per}}(\alpha,A)=\lim_{P\to\infty}\frac{1}{P}\log Z_P^{\mathrm{per}}(\alpha,A)=\log(\alpha (2\pi)^d)-\frac{1}{2\pi}\int_0^{2\pi}\log \det \Phi_A(e^{\ii\theta})d\theta.	
	\end{equation}
	Moreover, for any $k$, $l\in\setZ$ fixed,
	$$\lim_{P\to\infty}\EE[X_k^{(P),\mathrm{per}}X_l^{(P),\mathrm{per}}]=C_{l-k}(\Phi_A^{-1})$$
	where $C_n\left(\Phi_A^{-1}\right)$ is the $n$-th Fourier coefficient of $\Phi_A^{-1}$ for any $n\in\setZ$, see \eqref{fc}.
\end{theorem}
\begin{proof}
	This follows from a discrete Fourier change of variable and \eqref{free_energy_per} is a Riemann sum. See section \ref{fromtftomarkov} or chapter 13 of \cite{georgii2011gibbs} for more details.
\end{proof}

\paragraph{The function $\Phi_A$.}
As we will see, the function $\Phi_A$, that appears in the previous theorem, plays an important role in the model and the whole section \ref{phi} is dedicated to its complete study.

\paragraph{Via Invariant Measure Approach.} Another classical approach is to take advantage of the Markov property and look for left (or right) invariant measure. For instance, that is finding $B_L$ such that
$$\int_{\setC^d}\ee_{\beta_L,B_L}'(x)\ee_{\alpha,A}(x,y)dx=\Lambda\ee_{\beta_L,B_L}'(y)$$
for some eigenvalue $\Lambda$. We start with the following definition.
\begin{definition}[Left and right Schur-invariant matrices]
	Let $A\in H_{2d}^+(\setC)$. A matrix $B_L\in H_d^+(\setC)$ is left Schur-invariant for $A$ if it satisfies the following non linear equation
	\begin{equation}\label{leftinv}
		B_L=A_{RR}-A_{RL}(B_L+A_{LL})^{-1}A_{LR}
	\end{equation}
	where we recognize a Schur complement. Similarly, a matrix $B_R\in H_d^+(\setC)$  is right Schur-invariant for $A$ if it satisfies
	\begin{equation}\label{rightinv}
		B_R=A_{LL}-A_{LR}(A_{RR}+B_R)^{-1}A_{RL}.
	\end{equation}
\end{definition}
Standard Gaussian computations give the following theorem. 
\begin{theorem}[Left and right eigen-boundary conditions]\label{propre}
	Let $\alpha\in\setR_+^*$, $A\in H_{2d}^+(\setC)$ and $B_L$ (resp. $B_R$) left Schur-invariant (resp. right Schur-invariant) for $A$. Then
	\begin{equation}
		\int_{\setC^d}\ee_{\beta_L,B_L}'(x)\cdot \ee_{\alpha,A}(x,y)dx=\alpha (2\pi)^d \det (B_L+A_{LL})^{-1} \ee_{\beta_L,B_L}'(y)
	\end{equation}
	for all $y\in\setC^d$. In that case, we say that $\ee_{\beta_L,B_L}'$ is a left eigen-boundary density for $\ee_{\alpha,A}$.  Similarly,
	\begin{equation}
		\int_{\setC^d}\ee_{\alpha,A}(x,y)\cdot\ee'_{\beta_R,B_R}(y)dy =\alpha (2\pi)^d \det (A_{RR}+B_R)^{-1} \ee'_{\beta_R,B_R}(x).
	\end{equation}
	for all $x\in\setC^d$. In that case, we say that $\ee_{\beta_R,B_R}'$ is a right eigen-boundary density for $\ee_{\alpha,A}$. Moreover, the partition function of any homogeneous Gaussian Markov process of weight $\ee_{\alpha,A}$ with boundary conditions $(\ee_{\beta_L,B_L}',\ee_{\beta_R,B_R}')$ on any finite component $D$ of size $P\in\setN^*$ is 
	\begin{equation}
		Z_P(\alpha,A,\beta_L,B_L,\beta_R,B_R)=\Lambda^{P}\kappa
	\end{equation}
	with the eigenvalue $\Lambda=\alpha (2\pi)^d \det (B_L+A_{LL})^{-1} =\alpha (2\pi)^d \det (A_{RR}+B_R)^{-1} $and $\kappa=\langle \ee_{\beta_L,B_L}',\ee_{\beta_R,B_R}'\rangle_{L^2(\setC^d)}=1$ (with the good choices of $\beta_L$ and $\beta_R$). Then,
	\begin{equation}\label{free_energy_eigen}
		f^{\mathrm{eigen}}(\alpha,A)\coloneqq f(\alpha,A,\beta_L, B_L,\beta_R,B_R)=\log\Lambda.
	\end{equation} 
\end{theorem}
\begin{proof}
	This is a direct consequence of properties \ref{gluingboundwithweight}, \ref{gluingboundwithbound} and \ref{partitionfunction}, see section \ref{algebraicstructure} for algebraic details.
\end{proof}
This theorem first reduces the research of invariant measures to the understanding of the non-linear fix point relations \eqref{leftinv} and \eqref{rightinv}. Solutions are, in general, difficult to compute explicitly due to the non-linearity of these equations. Our main occupation will be to simplify these computations using a well-chosen decomposition. This theorem also shows that the computation of the partition function and the free energy is particularly simple under eigen-boundary conditions. Remark that, with this approach, the free energy is not a limit anymore.
\begin{rmk}\label{infinite_gibbs}
	Take $(D_n)_{n\in\setN}$ an increasing sequence of finite connected components such that 
	$$\bigcup_{n\in\setN}D_n=\setZ.$$
	The family of homogeneous Gaussian Markov processes $(X^{(D_n)})_{n\in\setN}$ of weight $\ee_{\alpha,A}$ with left and right eigen-boundary conditions defines a consistent family of probability measures. Kolmogorov extension theorem ensures the existence of a stochastic process $X$ on a certain probability space $(\Omega,\mathcal{F},\mathbb{P})$ whose marginal laws are the $X^{(D_n)}$. The law of $X$ is the infinite volume Gibbs measure associated to this problem.  
\end{rmk}  

\paragraph{What is the link?} Fourier transform and eigen-boundary conditions provide two methods for computing the same quantities. Nevertheless one doesn't easily see the link between these two approaches. Fourier transform is a global computation which totally hides local Markov property whereas invariant measure is a local computation based on the Markov property of the model. Moreover, Fourier transform provides explicit formulas (see theorem \ref{Fourier}) when invariant measure approach is based on equations \eqref{leftinv} and \eqref{rightinv}, which are difficult to solve explicitly. The main goal of this article is to reconcile these methods and to specify the underlying structure which describes the links that they share. By doing so, we will provide a simple algorithm for solving explicitly equations \eqref{leftinv} and \eqref{rightinv}.
\paragraph{A step towards higher dimension.} Establishing bridges between Fourier and eigen-boundary conditions for one dimensional Markov processes is interesting in itself but it becomes even more relevant in higher dimension. In higher dimension, Fourier transform approach keeps working and is well understood. Moreover, recent work from \cite{damien} uses operad theory to give a coherent notion of eigen-boundary ("invariant-boundary") conditions for higher dimensional Markov Processes. Due to geometrical properties, eigen-boundary objects are much more sophisticated than in dimension one and are not well understood yet. For higher dimensional Gaussian Markov processes, our understanding of Fourier method should be converted into the understanding of those new boundary objects via bridges between the two approaches. The present article can be viewed as a warm-up for upcoming work on higher dimensional Gaussian processes using this novel approach. 

\subsection{Main theorems and structure of the article}
We now expose the structure of the present paper. 
\subsubsection{Study of the meromorphic function $\Phi_A$}
As it can be seen in theorem \ref{Fourier}, the function $\Phi_A$ plays an important role: one can compute many quantities of interest from this single function. It is therefore fully studied in section \ref{phi}, which contains every relevant technical results. In particular, we will focus on the points of $\setC^*$ where $\Phi_A$  is not invertible, simply called the zeros of $\Phi_A$, see lemma \ref{zeros}. We will also construct two non trivial bases of $\setC^d$, indexed by zeros of $\Phi_A$, see definition \ref{basisindexedbyzeros}, along which many objects related to the model can be decomposed. This paper establishes a dictionary between zeros of  $\Phi_A$, together with their associated bases, and various probability laws of the model.

\subsubsection{From periodic boundary conditions to translation invariant Gaussian Markov on $\setZ$}
Section \ref{fromtftomarkov} gives more details for the periodic case and the structures behind theorem \ref{Fourier} obtained with Fourier transform. This is common material, for instance, see \cite{georgii2011gibbs}, chapter 13 for homogeneous Gaussian fields. We will compute marginal laws under thermodynamic limit, see \eqref{cvlaw}, and show that the Markov property is preserved, see theorem \ref{markov_prop}. This gives a natural candidate for the Hilbert space which "carries" the infinite volume Gibbs measure. 
\subsubsection{Markov and Schur complement: algebraic formulas}
Motivated by theorem \ref{propre}, geometric and algebraic structures of $(Z_P(\alpha,A))_{\alpha,A,P}$ or, equivalently, the ones given by weights $(\ee_{\alpha,A})_{\alpha,A}$, are studied in section \ref{algebraicstructure}. There is a natural associative product $m$ which allows to glue them together. Moreover, there is a non linear product $S$ at the level of Hermitian matrices  and a matrix $K$ such that we have the lift
$$m(\ee_{\alpha,A},\ee_{\Tilde{\alpha},\Tilde{A}})=\ee_{(2\pi)^d\alpha\Tilde{\alpha}\det K,S(A,\Tilde{A})},$$
see property \ref{schur_prod}. We will also have a look at algebraic properties of Gaussian boundary densities $\ee_{\beta,B}'$ on which Gaussian weights $\ee_{\alpha,A}$ naturally act. 

\subsubsection{Eigen-boundary conditions}
Section \ref{eigen-bound} focuses on the study of eigen-boundary conditions. In view of a higher dimension generalization, we wish to give a "dimension-free" recipe for the computation of $B_L$ and $B_R$ satisfying the non linear fix point relations \eqref{leftinv} and \eqref{rightinv}. This section aims at a better understanding of these equations and especially, at their translation in terms of $\Phi_A$(to establish links with Fourier). We briefly recall that we denote $\mathcal{S}(\Phi_A)=\mathcal{S}_{<1}(\Phi_A)\sqcup \mathcal{S}_{>1}(\Phi_A)$ the set of zeros of $\Phi_A$ inside and outside the unit disk, see lemma \ref{zeros} and equation \eqref{zerosinsideandoutside}. To each of these sets is associated a non-trivial basis of $\setC^d$, denoted $(u_w)_{w\in\mathcal{S}_{<1}(\Phi_A)}$ and $(u_w)_{w\in\mathcal{S}_{>1}(\Phi_A)}$, see definition \ref{basisindexedbyzeros}. 
We now expose the main theorem of the paper, proven in section \ref{eigen-bound}.
\begin{theorem}\label{main} Let $A\in H_{2d}^+(\setC)$ and $W_{>1}\in M_d(\setC)$ any diagonalizable matrix. Under assumptions \ref{full_rank} and \ref{multiplicite}, the following assertions are equivalent
	\begin{itemize}
		\item[(i)] The matrix $B_L=A_{RR}+A_{RL}W_{>1}^{-1}$ is left Schur-invariant for $A$.
		\item[(ii)] The matrix $W_{>1}$ is diagonal in the basis $(u_w)_{w\in\mathcal{S}_{>1}(\Phi_A)}$, with $W_{<1}u_w=wu_w$ for all $w\in\mathcal{S}_{>1}$.
		\item[(iii)] Let $\boldsymbol{x}\in\setC^d$. The sequence $(x_k)_{k=0}^{-\infty}$ with $x_k=W_{>1}^{-k}x_0$ is solution to disrete Dirichlet-type problem
		\begin{System}\label{dirich_1}A_{RL}x_{-k-1}+(A_{LL}+A_{RR})x_{-k}+A_{LR}x_{-k+1}=0 \text{ for all }k> 0\\
			x_0=\boldsymbol{x}\\
			\lim_{k\to -\infty}x_k=0
		\end{System}
	\end{itemize}
	Moreover, in those cases, the eigenvalue associated to $\ee'_{\beta_L,B_L}$ is 
	\begin{equation}\label{lambda}
		\Lambda=\alpha(2\pi)^d(-1)^d\frac{\prod_{w\in\mathcal{S}_{<1}}w}{\det(A_{RL})}
	\end{equation}
\end{theorem}
An analog theorem for right eigen-boundary condition $B_R$ involving zeros inside the unit disk $\mathcal{S}_{<1}(\Phi_A)$ is given in the same section, see theorem \ref{main2}. Here are a few comments concerning these two theorems. First of all, we remark from \eqref{lambda} that the eigenvalue is entirely determined by zeros of $\Phi_A$. By \eqref{free_energy_eigen}, taking the logarithm gives an expression for the free energy in terms of these zeros and which, as we will see in \eqref{egalite_energies_libres}, is equal to the free energy computed with periodic boundary conditions. Those theorems also give an explicit construction of left and right Schur-invariant elements $B_L$ and $B_R$ and reduce the eigen-boundary problem to a linear algebra problem. The construction of $B_L$ is performed from zeros of $\Phi_A$ outside the unit disk together with their associated basis whereas the construction of $B_R$ is performed from the zeros inside the unit disk. Finally we remark the emergence of new matrices $W_{<1}$ and $W_{>1}$ which solve the discrete harmonic problem $(iii)$ associated to the tridiagonal operator generated by $A$. In some sense, the function $\Phi_A$, and especially its zeros with their associated bases, is what links Fourier with invariant boundary conditions.

In the same section we will show that, under eigen-boundary conditions provided by these theorems, we recover the marginal laws computed through discrete Fourier transform, see lemma \ref{cov-inverse}. 

\paragraph{Interlude: solving invariant equations.} Apart from any probabilistic considerations, theorems \ref{main} and \ref{main2} provide an efficient algorithm to solve Schur-invariant equations \eqref{leftinv} and \eqref{rightinv}. Of course, these equations are easily solvable in the case $d=1$. It is just a matter of finding the roots of a second order polynomial with complex coefficients. Nevertheless, when $d>1$, this becomes quite difficult. We have to solve a non-linear system of $d(d+1)/2$ equations (by hermitian property) involving multivariate polynomials. Nevertheless, theorems \ref{main} and \ref{main2} give a very simple step by step algorithm to build a solution:
\begin{itemize}
	\item Find $\mathcal{S}(\Phi_A)$. This amounts to solve $\det(\Phi_A)=0$, that is finding the $2d$ roots a polynomial of order $2d$. 
	\item For any $w\in\mathcal{S}(\Phi_A)$, find a basis of $\ker\Phi_A(w)$. This can be performed by Gaussian elimination.
	\item Build $W_{>1}$ and $W_{<1}$ in the basis $(u_w)_{w\in\mathcal{S}_{>1}(\Phi_A)}$ and $(u_w)_{w\in\mathcal{S}_{<1}(\Phi_A)}$ respectively.
\end{itemize}

\subsubsection{Identities in law under Gibbs measure}
Under the infinite volume Gibbs measure provided by eigen-boundary conditions given by theorems \ref{main} and \ref{main2}, we give, in section \ref{lawid}, some identities in law under equilibrium. Motivated by the point $(iii)$ of theorem \ref{main} and \ref{main2}, we link the matrices $W_{<1}$ and $W_{>1}$ with conditional expectations (known to be harmonic). In particular, we will see that conditional expectations are given by two semi-groups generated by zeros of $\Phi_A$, see theorem \ref{condexp}. In this theorem, we will also compute conditional variances. 

\subsubsection{The case $A_{LR}$ non-invertible}
The main results of this paper are established under the assumption that $A_{LR}$ is full rank, see assumption \ref{full_rank}. In section \ref{non-invertible}, we show how to extend these results to the case where $A_{LR}$ is non-invertible. This does not change the nature of the results but makes computations much heavier. In this new context, the bases indexed by zeros of $\Phi_A$ now need to be completed with the bases of $\ker A_{LR}$  and $\ker A_{RL}$ to form basis of $\setC^d$. We will build left and right Schur-invariant boundary conditions in terms of matrices similar to $W_{<1}$ and $W_{>1}$ defined in theorem \ref{main} and \ref{main2}. These matrices are now diagonal in the bases indexed by zeros of $\Phi_A$ completed with the bases $\ker A_{LR}$ and $\ker A_{RL}$. Especially, they are no longer invertible (their kernels are respectively $\ker(A_{R
	L})$ and $\ker(A_{LR})$). 
\subsubsection{Application: probabilistic representation of Szegő limit theorem}
Szegő limit theorem is a famous result concerning the asymptotic behavior of determinant of Toeplitz matrices. In section \ref{Szego}, we propose an "eigen" version of this theorem in a simple case as an application. 
\paragraph{Acknowledgments.} We are very grateful to Damien Simon for his constant feedback and for the numerous fruitful discussions we had. 
\section{Study of the meromorphic function $\Phi_A$}\label{phi}
Let $d\in\setN^*$ and $A\in H_{2d}^+(\setC)$ with block decomposition previously given in \eqref{blockdecompo}. Let $\Phi_A$ be the matrix-valued function defined by
\begin{align}
	\Phi_A(z):\setC^{*} &\to M_d(\setC)\label{defphi}\\
	z & \mapsto \Phi_A(z)=A_{LL}+A_{RR}+A_{LR}z+A_{RL}z^{-1}. \nonumber
\end{align}
If $A$ is clear from the context, we skip the index and write simply $\Phi$. Note that, for all $z\in\setC^*$, $\Phi(z)$ is obtained from $A$ through
\begin{equation}\label{defphii}
	\Phi(z)=\begin{pmatrix}
		I_d\\
		\bar{z}^{-1}I_d
	\end{pmatrix}^{*}A\begin{pmatrix}
		I_d\\
		zI_d
	\end{pmatrix}
\end{equation}
with $I_d$ the $d$-dimensional identity matrix. This section is devoted to the study of this function $\Phi$. First, we remark that we have the following lemma. 
\begin{lemma}
	For all $z\in\setC^*$, 
	\begin{equation}\label{conj}
		\Phi\left(1/\Bar{z}\right)=\Phi(z)^*.  
	\end{equation}        
\end{lemma}
\begin{proof}
	It is a direct computation with the use of Hermitian property of $A$.
\end{proof}
This last lemma shows that the case $\vert z\vert=1$ is particular. We will therefore split the study of the function $\Phi$ into three subsections. Firstly, we will study $\Phi$ on the unit circle and deduce, from the positive definite and Hermitian properties, the existence of an inverse $\Phi^{-1}$. Moreover,we will show that the Fourier coefficients of this inverse follow a second order recursive relation. Secondly, we will study $\Phi$ outside the unit circle and focus on its zeros, the points where it is not invertible. In particular, we will see that the free energies computed via Fourier and via eigen-boundary conditions are both equal to a sum over these zeros. Finally, we will introduce two non trivial bases of $\setC^d$, indexed by zeros of $\Phi$. With the help of residue calculus, we will decompose $\Phi^{-1}$, everywhere it is invertible, and its Fourier coefficients on these bases. The bases indexed by zeros of $\Phi$ play an important role in the description of eigen-boundary conditions, see theorems \ref{main} and \ref{main2} and in probability laws under infinite volume Gibbs measure, see theorem \ref{condexp}.
\subsection{Behavior on the unit circle}
For this whole subsection, $d\in\setN^*$ and $A\in H_{2d}^+(\setC)$ are fixed.
\paragraph{Invertibility on the unit circle.}We start by showing that $\Phi$ is invertible on the unit circle.
\begin{lemma}[Hermitian property and positive definiteness]\label{Invertible}
	For all $\theta\in[0,2\pi]$, $\Phi(e^{\ii\theta})$ is Hermitian and positive definite of dimension $d$. Therefore, for all $\theta\in[0,2\pi]$, $\Phi(e^{\ii\theta})$ has a Hermitian and positive definite inverse $\Phi(e^{\ii\theta})^{-1}$. Finally, $\theta\mapsto\Phi(e^{\ii\theta})^{-1}$ is continuous.
\end{lemma} 
\begin{proof}
	Hermitian property is direct from \eqref{conj} and positive definite property follows from \eqref{defphii} together with positive definite property of $A$. Continuity is direct from continuity of $\Phi$ and inverse formula using adjugate matrix. 
\end{proof}
\paragraph{Fourier coefficients of $\Phi^{-1}$.} Motivated by the covariance matrix computed in theorem \ref{Fourier}, we are interested in the Fourier coefficients of $\Phi^{-1}$. We first recall the definition of matrix-valued Fourier coefficients.
\begin{definition}[Fourier coefficients of a matrix]
	Let $d\in\setN^*$ and, for all $\theta\in[0,2\pi]$, let  $\Psi(\theta)$ be a matrix of dimension $d\times d$. Suppose $\Psi_{i,j}\in L^{1}([0,2\pi])$ for all $i$, $j\in\set{1,\dots,d}$. For all $k\in\setZ$ we define the $k$-th Fourier coefficient of $\Psi$ as the matrix $C_{k}\in M_d(\setC)$ such that
	$$(C_{k}(\Psi))_{i,j}=\frac{1}{2\pi}\int_{0}^{2\pi}\left(\Psi(\theta)\right)_{i,j}e^{-\ii k \theta}d\theta$$
	for all $i$, $j\in\set{1,\dots,d}$. We will often write 
	
	\begin{equation}\label{fc}
		C_{k}(\Psi)=\frac{1}{2\pi}\int_{0}^{2\pi}\Psi(\theta)e^{-\ii k \theta}d\theta.   
	\end{equation}
\end{definition}
\paragraph{Notation.} As a continuous function, $\Phi^{-1}\in L^1([0,2\pi])$. In the following we will denote
\begin{equation}\label{C_k}
	C_k\coloneqq C_k(\Phi^{-1})
\end{equation}  
for all $k\in\setZ$ the Fourier coefficients of $\Phi^{-1}$. They obey to the following linear relation.
\begin{lemma}\label{recursive}
	For all $k\in\setZ$,  we have 
	\begin{equation}\label{hermf}
		C_k^*=C_{-k}   
	\end{equation}
	and 
	\begin{equation}\label{rec}
		C_{k}(A_{LL}+A_{RR})+C_{k+1}A_{RL}+C_{k-1}A_{LR}=I_{d}\delta_{k,0}.
	\end{equation}
\end{lemma}
\begin{proof}
	The first identity is due to Hermitian property of $\Phi^{-1}$. By definition of the inverse, for all $\theta\in[0,2\pi]$,
	$$\Phi(e^{\ii\theta})^{-1}\Phi(e^{\ii\theta})=I_{d}.$$
	Using Fourier decomposition $\Phi_{\theta}^{-1}=\sum_{k\in\setZ}C_{k}e^{\ii k \theta}$ and the definition of $\Phi_{\theta}$ we get
	$$\sum_{k\in\setZ}\left(C_{k}(A_{LL}+A_{RR})+C_{k+1}A_{RL}+C_{k-1}A_{LR}\right)e^{\ii k \theta}=I_{d}$$
	and uniqueness of Fourier decomposition gives the result.
\end{proof}
\begin{rmk}
	Fourier coefficients $(C_k)_{k\in\setZ}$ follow a second order recursive relation. This can be seen as a  consequence of the Markov property of the model or, in other words, of the "nearest neighbor" structure of the model. This can be seen in the fact that $Q_{(A,B_\bullet)}^{(D)}$ is tridiagonal, see \eqref{matrice_couplage}: it is a "nearest neighbor" operator, generated by $A$ with boundary conditions $B_L$ and $B_R$.
\end{rmk}
\begin{rmk}
	Taking Hermitian transpose of \eqref{rec} together with \eqref{hermf} we also get,
	$$(A_{LL}+A_{RR})C_k+A_{RL}C_{k+1}+A_{LR}C_{k-1}=I_d \delta_{k,0}$$
	for all $k\in\setZ$.
\end{rmk}

\subsection{Behavior on $\setC^*$}
For this whole subsection, $d\in\setN^*$ and $A\in H_{2d}^+(\setC)$ are fixed. Except if the contrary is mentioned, we also work under assumption \ref{full_rank}. The matrix $\Phi$ is not invertible everywhere on $\setC^*$. The following lemma describes the singularities of $\Phi^{-1}$.

\begin{lemma}\label{zeros}
	Suppose assumption \ref{full_rank} fulfilled. The equation $\det\Phi(z)=0$ for $z\in\setC^*$ admits $2d$ solutions counted with multiplicity. We let
	\begin{equation}\label{zeros_of_phi}
		\mathcal{S}(\Phi)=\lbrace z\in\setC^*; \det\Phi(z)=0\rbrace,
	\end{equation}
	abusively called the set of zeros of $\Phi$. For all $w\in\mathcal{S}(\Phi)$, $1/\overline{w}\in\mathcal{S}(\Phi)$. Moreover, for all $z\in\setC^*$,
	\begin{equation}\label{polynomial}
		\det\Phi(z)=\det(A_{LR})z^{-d}\prod_{w\in\mathcal{S}(\Phi)}(z-w).
	\end{equation}
\end{lemma}
\begin{proof}
	Remark that $z\Phi(z)$ is a matrix polynomial of degree $2$. 
	Coming back to the definition of determinant we get that, for all $z\in\setC^*$, 
	$$z^d\det\Phi(z)=\det z\Phi(z)$$
	so that
	$$z^d\det\Phi(z)=\sum_{\sigma\in\mathfrak{S}_d}\epsilon(\sigma)\prod_{i=1}^d\left(A_{LR}z^2+(A_{LL}+A_{RR})z+A_{RL}\right)_{i,\sigma(i)}$$
	is a polynomial of degree $2d$ and $0$ is not a root since $A_{RL}$ is invertible. Therefore, counted with multiplicity, the equation $\det\Phi(z)=0$ admits $2d$ solutions on $\setC^*$. It is clear that the coefficient in front of $z^{2d}$ is $\det(A_{LR})$ so we get \eqref{polynomial}. Remark that the constant term is $\det(A_{RL})$. Finally, take $w\in\mathcal{S}(\Phi)$, the identity \eqref{conj} gives $1/\overline{w}\in\mathcal{S}(\Phi)$.
\end{proof}
A direct consequence of the latter proof is 
\begin{equation}\label{prod_res}
	\frac{\prod_{w\in\mathcal{S}_{<1}}w}{\det(A_{RL})}=\frac{\prod_{w\in\mathcal{S}_{<1}}\overline{w}}{\det(A_{LR})}.
\end{equation}
\paragraph{Notation.} Since $A$ is fixed, so is $\Phi$, and we will often simply write $\mathcal{S}$ instead of $\mathcal{S}(\Phi)$. We can decompose $\mathcal{S}=\mathcal{S}_{< 1}\sqcup\mathcal{S}_{> 1}$ with
\begin{equation}\label{zerosinsideandoutside}
	\mathcal{S}_{<1}=\lbrace w\in\mathcal{S}; \vert w\vert<1\rbrace\text { and }\mathcal{S}_{>1}=\lbrace w\in\mathcal{S}; \vert w\vert>1\rbrace	
\end{equation}
zeros inside and outside the unit disk. We recall that there is no zero on the unit circle thanks to lemma \ref{Invertible}. 
\begin{rmk}
	Previous lemma tells that $\Phi$ is not invertible on $d$ points inside the unit disk and on $d$ points outside counted with multiplicity. In particular, $\Phi$ admits the following Birkhoff decomposition
	$$\Phi=(\Phi^+)^*\Phi^+$$
	with $\Phi^+$ continuous and invertible on the unit disk, see \cite{clancey2013factorization}. We will however use a more linear version, see theorem \ref{residues}, involving bases indexed by zeros of $\Phi$ inside and outside the unit disk. 
\end{rmk}

\paragraph{Link with free energy.} As an application of the previous lemma, and to highlight the role played by $\mathcal{S}$, we propose to show that the free energy computed with periodic boundary conditions $f^{\mathrm{per}}(\alpha,A)$ (see \eqref{free_energy_per}) and the free energy computed with eigen-boundary conditions $f^{\mathrm{eigen}}(\alpha,A)$ (see \eqref{free_energy_eigen}) are equal and can be expressed in terms of zeros of $\Phi$. This is a first link between the two approaches. Nevertheless, to link the Fourier approach with the invariant matrices $B_L$ and $B_R$ we will need more than just the zeros of $\Phi$: we will need the bases indexed by zeros of $\Phi$, see definition \ref{basisindexedbyzeros} and theorems \ref{main} and \ref{main2}. 
\begin{lemma} Let $\alpha>0$. Under assumptions \ref{full_rank} and \ref{multiplicite}, we have
	\begin{equation}\label{egalite_energies_libres}
		f^{\mathrm{per}}(\alpha,A)=f^{\mathrm{eigen}}(\alpha,A).
	\end{equation}	
\end{lemma}

\begin{proof}
	Using theorem \ref{main} (proven in section \ref{eigen-bound})	and equations \eqref{free_energy_per} and \eqref{free_energy_eigen}, we only have to show that 
	\begin{equation}\label{val_propre_residues}
		-\frac{1}{2\pi}\int_0^{2\pi}\log \det \Phi(e^{\ii\theta})d\theta=\log\left((-1)^d\frac{\prod_{w\in\mathcal{S}_{<1}}w}{\det(A_{RL})}\right).
	\end{equation}	
	For all $\theta\in[0,2\pi]$, using \eqref{polynomial} we have,
	\begin{align*}
		\det\Phi(e^{\ii\theta})&=\det(A_{LR})e^{-\ii d\theta}\prod_{w\in\mathcal{S}_{<1}}(e^{\ii\theta}-w)\prod_{w\in\mathcal{S}_{<1}}(e^{\ii\theta}-\frac{1}{\overline{w}})\\
		&=(-1)^d\frac{\det(A_{LR})}{\prod_{w\in\mathcal{S}_{<1}}\overline{w}}\prod_{w\in\mathcal{S}_{<1}}(1-we^{-\ii\theta})\prod_{w\in\mathcal{S}_{<1}}(1-\overline{w}e^{\ii\theta}).
	\end{align*}
	In particular, since $\det\Phi(e^{\ii\theta})>0$ for all $\theta\in[0,2\pi]$ by lemma \ref{Invertible}, we have $$(-1)^d\frac{\det(A_{LR})}{\prod_{w\in\mathcal{S}_{<1}}\overline{w}}>0.$$ Denote $S^1$ the unit circle, for all $u\in\setC^*$ such that $\vert u \vert<1$,
	\begin{align*}
		\int_0^{2\pi}\log(1-ue^{\ii\theta})d\theta=\int_{\mathcal{C}_1}\frac{1}{\ii z}\log(1-uz)dz=0.
	\end{align*}
	Indeed, $\log$ is the principal value, continuous onto $\mathbb{D}(1,\vert u\vert)$  the disk centered at $1$ with radius $\vert u\vert$. Therefore, the function $\frac{1}{\ii z}\log(1-uz)$ as a unique pole on $\mathbb{D}(0,1)$, the disk centered at $0$ with radius $1$, and the associated residue is equal to $0$. A change of variable also gives
	$$\int_0^{2\pi}\log(1-ue^{-\ii\theta})d\theta=0.$$
	Take $u=w$ or $u=\overline{w}$ and the result follows from the fact that logarithm of product is the sum of logarithms up to a multiple of $2\ii\pi$ and from \eqref{prod_res}. 
\end{proof} 

\paragraph{Symmetry property.} We end this section by giving little results concerning symmetry properties of $\Phi$. This will be convenient for deducing theorem \ref{main} from theorem \ref{main2} with a symmetric argument, see section \ref{eigen-bound}.  Letting $I_d$ the $d$-dimensional identity matrix, $S=\begin{pmatrix}
	0 & I_d\\
	I_d & 0
\end{pmatrix}$ acts by conjugation on $A$ in the following way
$$SAS^{-1}=\begin{pmatrix}
	A_{RR} & A_{RL}\\
	A_{LR} & A_{LL}
\end{pmatrix}$$
which reverses direction of edges. For all $z\in\setC^*$, we have
$$\Phi_{SAS^{-1}}(z)=\Phi_A\left(\frac{1}{z}\right).$$
We deduce the following result.
\begin{lemma}\label{leftright}
	Let $A\in H_{2d}^+(\setC)$. For all $z\in\mathcal{S}(\Phi_A)$, $\ker\Phi_{SAS^{-1}}(\Bar{z})=\ker\Phi_{A}\left(1/\Bar{z}\right).$ Moreover, we have
	$$\mathcal{S}(\Phi_{SAS^{-1}})=\overline{\mathcal{S}(\Phi_A)}$$
	where $\overline{\mathcal{S}(\Phi_A)}=\lbrace \Bar{z}, z\in\mathcal{S}(\Phi_A)\rbrace$.  
\end{lemma}
\begin{proof}
	The first identity is obvious thanks to the previous remark. If $z\in\mathcal{S}(\Phi_A)$ then $1/\Bar{z}\in\mathcal{S}(\Phi_A)$ and,  thanks to the previous remark, $\Bar{z}\in\mathcal{S}(\Phi_{SAS^{-1}})$.
\end{proof}

\subsection{The bases indexed by zeros of $\Phi$.}
Once again, for this whole subsection $d\in\setN^*$ and $A\in H_{2d}^+(\setC)$ are fixed. We will now expose an important notion of this paper, the bases indexed by zeros of $\Phi$. These bases are involved in most of the interesting quantities of the model, especially when looking for left and right Schur-invariant boundaries, see theorems \ref{main} and \ref{main2}. Besides the fact that we work under assumption \ref{full_rank}, we also make the following one. 
\begin{assumptionp}{II}\label{multiplicite}
	For any root $w$ of the polynomial $z^d\det\Phi(z)$, studied in lemma \ref{zeros}, let $\mathrm{mult}(w)$ be its multiplicity. To make things easier, we make the following assumption
	\begin{equation}
		\mathrm{mult}(w)=\dim\ker\Phi(w)=1 
	\end{equation}
	for all $w\in\mathcal{S}$. In that case, $\vert \mathcal{S}\vert = 2d$ and has exactly $d$ non zero elements inside the unit disk and $d$ outside. For all $w\in\mathcal{S}$, we fix $u_w\neq 0$ such that
	\begin{equation}
		\Phi(w)=\mathrm{Vect}_\setC(u_w).
	\end{equation}
\end{assumptionp}
Additional work might be needed for the general case. We only give an example where $\dim\ker\Phi(w)>1$.
\paragraph{Example 1.} 
Suppose $A_{LL}=A_{RR}=I_d$ and $A_{LR}$ Hermitian, thus diagonalized by a unitary matrix. Write $(a_i)_{i=1}^d$ its real eigenvalues. In that case, up to a change of basis, for all $z\in\setC^*$, $\Phi(z)$ is diagonal with diagonal entries given by $\phi_i(z)=1+a_iz+a_iz^{-1}$ for $i\in\lbrace1,\dots,d\rbrace$. If any eigenvalue is of multiplicity greater than $1$, the associated kernel is of dimension greater than $1$. 
\begin{definition}[Bases indexed by zeros of $\Phi$]\label{basisindexedbyzeros}  Under assumptions \ref{full_rank} and \ref{multiplicite}, we prove in theorem \ref{residues} that the families $(u_w)_{w\in\mathcal{S}_{<1}}$ and $(u_w)_{w\in\mathcal{S}_{>1}}$ are bases of $\setC^d$. They are called the bases indexed by zeros of $\Phi$, respectively, inside the unite disk and outside the unit disk.
\end{definition}
We introduce the following notation.
\paragraph{Notation.} Every $x\in\setC^d$ can be decomposed in the basis $(u_w)_{w\in\mathcal{S}_{<1}}$, resp. $( u_w)_{w\in\mathcal{S}_{>1}}$, as
\begin{equation}\label{decomposition_dans_base_zeros}
	\sum_{w\in\mathcal{S}_{<1}}x(w) u_w \text{ resp. } \sum_{w\in\mathcal{S}_{>1}}x(w) u_w.
\end{equation}
Before giving the main theorem of this section, we start by giving the following lemma, which is direct from \eqref{conj}. 
\begin{lemma}\label{rmksym}
	Let $w\in\mathcal{S}$. Then, $u_w\in\ker\Phi(w)$ if and only if $u_{w}\in\mathrm{Im}\Phi(1/\overline{w})^\perp$. 
\end{lemma}
The next theorem decomposes $\Phi^{-1}$ along the vectors indexed by zeros of $\Phi$ which are proven to form bases of $\setC^d$.
\begin{theorem}\label{residues} Assume \ref{full_rank} and \ref{multiplicite}. For all $w\in\mathcal{S}$, $\langle u_{1/\overline{w}}, \Phi'(w) u_w \rangle\neq 0$. Then let $\alpha_w\in\setC^*$ and $P_w\in M_d(\setC)$ defined as
	\begin{equation*}
		\alpha_w=\frac{1}{\langle u_{1/\overline{w}}, \Phi'(w) u_w \rangle } \text{ and } P_wx=\langle u_{1/\overline{w}},x\rangle u_w
	\end{equation*}	
	$w\in\mathcal{S}$ and for all $x\in\setC^d$. For all $z\in\setC^* \setminus \mathcal{S}$, we have
	\begin{align}
		\Phi(z)^{-1}=&\sum_{w\in\mathcal{S}}\frac{\alpha_w}{z-w}P_w\label{eq1}\\
		=&\sum_{w\in\mathcal{S}_{<1}}\left(\frac{\alpha_w}{z-w}P_w+ \frac{\alpha_{1/\overline{w}}}{z-1/\overline{w}}P_{1/\overline{w}}\right)\label{eq2}
	\end{align}
	Moreover, $(u_w)_{w\in\mathcal{S}_{<1}}$ and $(u_w)_{w\in\mathcal{S}_{>1}}$ are bases of $\setC^d$. 
\end{theorem}
\begin{proof}
	The proof is based on residue calculus. First, recall
	\begin{equation}\label{adjugate}
		\Phi(z)^{-1}=\frac{1}{\det\Phi(z)}\mathrm{adj}(\Phi(z))
	\end{equation}
	for all $z\in\setC^*\setminus\mathcal{S}$, where $\mathrm{adj}(\Phi(z))$ is the adjugate matrix of $\Phi(z)$. We know that $\det\Phi(z)$ is a linear combination of $z^i$ for $i\in\set{-d,\dots,d}$ and the coefficients of order $d$ and $(-d)$ are non zero under assumption \ref{full_rank}, due to lemma \ref{zeros}. Since $\mathrm{adj}(\Phi(z))$ is a linear combination of $z^i$ with $i\in\set{-d+1,\dots,d-1}$, it comes
	$$\lim_{\vert z \vert \to 0}\Phi(z)^{-1}=0 \text{ and }\lim_{\vert z \vert \to \infty}\Phi(z)^{-1}=0.$$
	We let $\mathcal{C}_R$ be the circle centered at origin of radius $R>0$. Take $R$ sufficiently large so that $\mathcal{S}\subset \mathbb{D}(0,R)$, the disk centered at origin of radius $R$.  For any $z\in\setC^*\setminus \mathcal{S}$, we have
	\begin{align*}
		\int_{\mathcal{C}_R}\Phi(u)^{-1}\frac{1}{u-z}du &=2\pi\ii\left(\Phi(z)^{-1}+\sum_{w\in\mathcal{S}\cup\set{0}}\mathrm{Res}(\Phi^{-1},w)\frac{1}{w-z}\right)\\
		&=2\pi\ii\left(\Phi(z)^{-1}+\sum_{w\in\mathcal{S}}\mathrm{Res}(\Phi^{-1},w)\frac{1}{w-z}\right)
	\end{align*}
	where we used the limit at $0$ previously computed. Moreover, using the limit of $\Phi(u)^{-1}$ when $\vert u \vert\to\infty$, we get 
	$$\int_{\mathcal{C}_R}\Phi(u)^{-1}\frac{1}{u-z}du=\mathrm{Res}\left(\Phi^{-1}\frac{1}{\bullet-z},\infty\right)=0.$$
	We deduce,
	$$\Phi(z)^{-1}=\sum_{w\in\mathcal{S}}\mathrm{Res}(\Phi^{-1},w)\frac{1}{z-w}.$$
	To get \eqref{eq1}, we only have to show that, for all $w\in\mathcal{S}$,
	$$\mathrm{Res}(\Phi^{-1},w)=\lim_{u\to w}(u-w)\Phi(u)^{-1}=\alpha_wP_w.$$
	Fix any $w\in\mathcal{S}$. We are therefore interested in a Taylor expansion of \eqref{adjugate}. We first have, for any $u\to w$,
	\begin{equation}\label{DL_det}
		\det\Phi(u)=(\det\Phi(w))'(u-w)+o(u-w)
	\end{equation}
	and we check that $(\det(\Phi(w))'\neq 0$ as $w$ is assumed to be of multiplicity $1$ by assumption \ref{multiplicite}. We also have that, when $u\to w$,
	\begin{equation}\label{DL_adj}
		\mathrm{adj}(\Phi(u))=\mathrm{adj}(\Phi(w))+(\mathrm{adj}(\Phi(w)))'(u-w)+o(u-w).
	\end{equation}
	We need to check that none of these terms is zero. We have, for any $z\in\setC^*$,
	\begin{equation}\label{formule_adj}
		\Phi(z)\mathrm{adj}(\Phi(z))=\det(\Phi(z))I_d,\text{ and }\mathrm{adj}(\Phi(z))\Phi(z)=(\det\Phi(z))I_d.
	\end{equation}
	It first gives
	$$\Phi(z)\mathrm{adj}(\Phi(w))=\mathrm{adj}(\Phi(w))\Phi(w)=0.$$
	So that, by assumption \ref{multiplicite}, and lemma \ref{rmksym}, $\mathrm{adj}(\Phi(w))$ is of form $\mathrm{adj}(\Phi(w))=\beta_wP_w$ with $\beta_w\in\setC$. Now, taking the derivative of \eqref{formule_adj} in $w$ gives
	$$\beta_w\Phi'(w)P_w+\Phi(w)(\mathrm{adj}(\Phi(w)))'=(\det\Phi(w))' I_d.$$
	This is the sum of two linear maps: the first one is of rank at most $1$ and the second one of rank at most $d-1$. It forces $\beta_w\neq 0$ and $(\mathrm{adj}(\Phi(w)))'\neq 0$. Now, multiplying this identity on the left by $u_{1/\overline{w}}^*$ and using lemma \ref{rmksym}, gives for all $x\in\setC^d$,
	$$\beta_w\langle u_{1/\bar{w}}, \Phi'(w)u_w \rangle \langle u_{1/\bar{w}}, x \rangle=(\det\Phi(w))'\langle u_{1/\bar{w}},x\rangle$$
	so that $\langle u_{1/\overline{w}}, \Phi'(w) u_w \rangle\neq 0$ and $\beta_w=\alpha_w(\det\Phi(w))'$. Combining \eqref{DL_det} and \eqref{DL_adj} with \eqref{adjugate}, we end up with
	$$\Phi(u)^{-1}=\frac{1}{u-w}\alpha_wP_w+\frac{1}{(\det\Phi(w))'}(\mathrm{adj}(\Phi(w)))'+o(1)$$
	as $u\to w$ and \eqref{eq1} is proved. The second expression \eqref{eq2} follows from a change of variable in the sum \eqref{eq1}. We finally show that $(u_w)_{w\in\mathcal{S}_{<1}}$ is a basis. If $(u_w)_{w\in\mathcal{S}_{<1}}$ is not a basis, there exists $x\in(\mathrm{Vect}((u_w)_{w\in\mathcal{S}_{<1}}))^{\perp}$ non-zero so that for all $\theta\in[0,2\pi]$,
	\begin{align*}
		\langle x,\Phi(e^{\ii\theta})^{-1} x\rangle =\sum_{w\in\mathcal{S}_{<1}}\left( \frac{\alpha_w}{e^{\ii\theta}-w}\langle x, P_wx\rangle+\frac{\alpha_{1/\overline{w}}}{e^{\ii\theta}-1/\overline{w}}\langle x, P_{1/\overline{w}}x\rangle\right)  =0
	\end{align*}
	which is in contradiction with positive definiteness of $\Phi(e^{\ii\theta})^{-1}$. Then, $(u_w)_{w\in\mathcal{S}_{<1}}$ is a basis of $\setC^d$. With similar arguments, so is $(u_w)_{w\in\mathcal{S}_{>1}}$.
\end{proof}

We recall that the Fourier coefficients of $\Phi^{-1}$ are denoted $(C_k)_{k\in\setZ}$, see \eqref{C_k}. An application of the latter lemma is the decomposition of these Fourier coefficients along the bases indexed by zeros of $\Phi$.
\begin{corollary}\label{expliFC} Under assumptions and notations of theorem \ref{residues}, for all $k\geq 0$, we have
	\begin{equation}
		C_{k}=\sum_{w\in\mathcal{S}_{>1}}-\frac{\alpha_w}{w^{k+1}} P_w\text{ and }C_{-k}=\sum_{w\in\mathcal{S}_{<1}}w^{k-1}\alpha_wP_w.
	\end{equation}
	
\end{corollary}
\begin{proof}
	Let $k\geq 0$, and denote $S^1$ the complex unit circle. We have, from theorem \ref{residues},
	\begin{align*}
		C_{-k}=\frac{1}{2\ii\pi}\sum_{w\in\mathcal{S}}\alpha_wP_w\int_{\mathcal{C}_1}\frac{z^{k-1}}{z-w}dz.
	\end{align*}
	A little care for the case $k=0$ and residues theorem give the result. The identity for $C_k$ 	is deduced from using $C_k=C_{-k}^*$, see \eqref{hermf}, together with the identity
	$$\Phi'(w)w^2=-\Phi'\left(1/\bar{w}\right)^*$$
	derived from \eqref{defphi} and using hermitian property of $A$.
\end{proof}
\begin{rmk}\label{decrease}
	As mentioned in theorem \ref{Fourier}, the Fourier coefficients $(C_k)_{k\in\setZ}$ give the covariance matrix in thermodynamic limit. The latter corollary ensures that these covariances decrease exponentially. 
\end{rmk}
\paragraph{Short conclusion.} We introduced $(u_w)_{w\in\mathcal{S}_{<1}}$ and $(u_w)_{w\in\mathcal{S}_{>1}}$ two families that parametrize $\Phi^{-1}$ and its Fourier coefficients $(C_k)_{k\in\setZ}$. Other objects related to the model  will be decomposed along these families in sections \ref{eigen-bound} and \ref{lawid}. 
\section{From periodic boundary conditions to translation invariant Gaussian Markov on $\setZ$}\label{fromtftomarkov}
Motivated by further comparisons with the eigen-boundary case, this section deals with basics results concerning the periodic case and structures behind the theorem \ref{Fourier}.  We will see that the matrices involved when considering periodic Gaussian Markov processes are block-circulant and therefore diagonal in the basis of Fourier. This allows an easy computation of the correlation function in the thermodynamic limit. Moreover, this gives a natural candidate for the Hilbert space which carries the infinite dimensional Gaussian process on $\setZ$ which, as it will be shown, verifies the Markov property.
\subsection{Circulant and multiplication operators.}\label{section_circulant_multiplication}
We start by recalling the definitions of block-circulant and multiplication operators. Let $d$, $P\in\setN^*$. In the following, we will be considering the Hilbert spaces $(\setZ/P\setZ\to\setC^d,\langle \cdot, \cdot \rangle_{\setC^{dP}})$ and $(l^2(\setZ\to\setC^d),\langle\cdot,\cdot\rangle_{l^2})$ with $\langle \cdot, \cdot \rangle_{\setC^{dP}}$ and $\langle\cdot,\cdot\rangle_{l^2}$ usual inner products. In these contexts, a $dP$-dimensional block-matrix $B$ (resp. an infinite dimensional block operator $B$) with $d$-dimensional blocks is
$$B=(B_{k,l})_{k,l\in\lbrace 0,\dots,P-1\rbrace}\text{ (resp. }B=(B_{k,l})_{k,l\in\setZ})$$
with $B_{k,l}\in M_{d}(\setC)$ for all $k$ and $l$. 
\paragraph{Circulant and multiplication operators on $\setZ/P\setZ$.} Let $\Psi: S^1\to M_d(\setC)$ continuous. Its discrete Fourier coefficients are defined in the following way
\begin{equation}
	C_k^{(P)}(\Psi)=\frac{1}{\sqrt{P}}\sum_{i=0}^{P-1}\Psi(\omega_P^i)\omega_P^{-ki}\in M_d(\setC)
\end{equation}
for all $k\in\lbrace 0,\dots,P-1\rbrace$. This is, of course, the discrete version of the Fourier coefficients introduced in \eqref{fc}. We denote $C^{(P)}(\Psi)=(C_{l-k}^{(P)}(\Psi))_{k,l\in\lbrace 0,\dots,P-1\rbrace}\in M_{dP}(\setC)$ the block-circulant matrix generated by discrete Fourier coefficients of $\Psi$.  Such a matrix is block-diagonalizable via discrete Fourier transform into the multiplication operator on $\setZ/P\setZ$ generated by $\Psi$:
\begin{align*}
	M^{(P)}(\Psi):\left(\setZ/P\setZ\to\setC^d\right)&\to \left(\setZ/P\setZ\to\setC^d\right)\\
	(x_k)_{k=0}^{P-1}&\mapsto\left(\Psi(\omega_{P}^{k})x_k\right)_{k=0}^{P-1}.
\end{align*}
\paragraph{Circulant and multiplication operators on $\setZ$.} Let $\Psi: S^1\to M_d(\setC)$ continuous and  for all $k\in\setZ$, $C_k(\Psi)\in M_d(\setC)$ its $k$-th Fourier coefficient already defined in \eqref{fc}.  We denote $C(\Psi)=(C_{l-k}(\Psi))_{k,l\in\setZ}$ the infinite dimensional block-circulant operator acting on $l^2(\setZ\to\setC^d)$. This operator, via Fourier correspondence between $l^2(\setZ\to\setC^d)$ and $L^2(S^1\to\setC^d)$, is linked to
\begin{align*}
	M(\Psi):L^2(S^1\to\setC^d)&\to L^2(S^1\to\setC^d)\\
	f &\mapsto\Psi f
\end{align*}
the multiplication operator on $L^2(S^1\to\setC^d)$ generated by $\Psi$.
\subsection{Gaussian Markov processes on $\setZ/P\setZ$}
When working with periodic boundary conditions, the matrix appearing in the density of the Gaussian Markov process is (tridiagonal) block-circulant and is therefore block-diagonalizable using discrete Fourier transform. We first briefly summarize this fact.

Let $A\in H_{2d}^+(\setC)$ fixed and $\Phi\coloneqq \Phi_A$, see section \ref{phi}. A Gaussian Markov process, on the periodic lattice $\setZ/P\setZ$, $X^{(P),\mathrm{per}}=(X_k^{(P),\mathrm{per}})_{k=0}^{P-1}$ has density proportional to
\begin{equation}\label{periodic_density}
	g_{X^{(P),\mathrm{per}}}(x^{(P)})\propto e^{-\frac{1}{2}(x^{(P)})^*Q_{A,\mathrm{per}}^{(P)}x^{(P)}}
\end{equation}
for all $x^{(P)}=(x_k^{(P)})_{k=0}^{P-1}\in\setC^{dP}$ where $Q_{A,\mathrm{per}}^{(P)}=C^{(P)}(\Phi)$ the $dP$-dimensional block-circulant matrix generated by $\Phi$:
\begin{equation}
	Q_{A,\mathrm{per}}^{(P)}= \begin{pmatrix}
		T & A_{LR} & 0 & 0 & \dots & 0 & A_{RL} \\
		A_{RL} & T & A_{LR} & 0 & \dots & 0 & 0\\
		0 & A_{RL} & T & A_{LR} & \ddots & 0 & 0\\
		\vdots & \ddots & \ddots & \ddots & \ddots & \vdots & \vdots\\
		0 & \dots & 0 & \ddots & T & A_{LR} & 0\\
		A_{LR} & 0 & \dots & 0 & 0 & A_{RL}& T
	\end{pmatrix},
\end{equation}
with $T=A_{LL}+A_{RR}$. Recalling 
$$\widehat{X}_k^{(P)}=\frac{1}{\sqrt{P}}\sum_{i\in\setZ/P\setZ}X_i^{(P),\mathrm{per}}\omega_{P}^{ik},$$
we recover the result of the theorem \ref{Fourier}, which states that $\widehat{X}^{(P)}=(\widehat{X}_k^{(P)})_{k=0}^{P-1}$ has density proportional to
$$g_{\widehat{X}^{(P)}}(x^{(P)})\propto e^{-\frac{1}{2}(x^{(P)})^*M^{(P)}(\Phi)x^{(P)}}.$$
for all $x^{(P)}=(x_k^{(P)})_{k=0}^{P-1}\in\setC^{dP}$. We end up with the following lemma. 
\begin{lemma}\label{periodic}
	The following assertions are equivalent
	\begin{itemize}
		\item[i)] $X^{(P),\mathrm{per}}$ is a centered Gaussian Markov process on $\setZ/P\setZ$ of weight $\ee_{\alpha,A}$.
		\item[ii)] $X^{(P),\mathrm{per}}$ is a centered Gaussian process on $H^{(P)}=(\setZ/P\setZ\to\setC^d,\langle \cdot, C^{(P)}(\Phi^{-1})\cdot\rangle_{\setC^{dP}})$. In other words, for any $E\subset \setZ/P\setZ$ and $(\alpha_k)_{k\in E}$ $\setC^d$-valued sequence, we have
		$$\mathbb{E}\left[\exp\left(\ii\sum_{k\in E}\mathrm{Re }\langle \alpha_k, X_k^{(P),\mathrm{per}} \rangle_{\setC^d}\right)\right]=\exp\left(-\frac{1}{2}\sum_{(k,l)\in E^2}\langle\alpha_k,C_{l-k}^{(P)}(\Phi^{-1})\alpha_l\rangle_{\setC^d}\right).$$
		\item[iii)] $\widehat{X}^{(P)}$ is a centered Gaussian process on $\widehat{H}^{(P)}=(\setZ/P\setZ\to\setC^d,\langle \cdot, M^{(P)}(\Phi^{-1})\cdot\rangle_{\setC^{dP}})$.
	\end{itemize}
\end{lemma}
Covariance matrices of $X^{(P),\mathrm{per}}$ and $\widehat{X}^{(P)}$and their inverses are summarized in the following commutative diagram 
\begin{equation}\label{periodicgdiagram}
	\begin{tikzcd}
		Q_{A,\mathrm{per}}^{(P)} \arrow[r, "^{-1}"] \arrow[d, "\mathrm{DFT}"]
		& C^{(P)}(\Phi^{-1})\arrow[l]\arrow[d,"\mathrm{DFT}"] \\
		M^{(P)}(\Phi)\arrow[r,"^{-1}"]\arrow[u]
		& M^{(P)}(\Phi^{-1})\arrow[l]\arrow[u]
	\end{tikzcd}
\end{equation} 

where "DFT" stands for Discrete Fourier Transform. Upper line (resp. lower) corresponds to
$X^{(P),\mathrm{per}}$ (resp. $\widehat{X}^{(P)})$ and left column (resp. right) corresponds to inverse of covariance matrices (resp. covariance matrices). 
\begin{rmk}
	We insist that , for any $P$ there exists such  a diagram and Hilbert spaces, $H^{(P)}$ and $\widehat{H}^{(P)}$. However, none of them are connected either by any projection nor any restriction. See \eqref{eigendiagram} for a similar discussion under eigen-boundary conditions instead of periodic boundary conditions.
\end{rmk}
\subsection{Thermodynamic limit}
Now that the case $\setZ/P\setZ$ is understood, we compute the marginal laws as $P\to\infty$ and discuss the Hilbert space obtained at the limit.  
\paragraph{Convergence in law for finite dimensional marginal laws.}
We introduce
\begin{align*}
	(\setZ\to\setC^d)&\to(\setZ/P\setZ\to\setC^d)\\
	(\alpha_k)_{k\in\setZ}&\mapsto (\alpha_k^{(P)})_{k=0}^{P-1}=\begin{cases}
		\alpha^{(P)}_k=\alpha_{k} \text{ if } k\in [0,P/3[\\
		\alpha^{(P)}_{P-k}=\alpha_{-k} \text{ if }k\in [1,P/3[\\
		\alpha^{(P)}_k=0 \text{ elsewhere.}
	\end{cases}
\end{align*}
and, with abuse of notation, we let $X_\alpha^{(P),\mathrm{per}}=X_{\alpha^{(P)}}^{(P),\mathrm{per}}\coloneqq \sum_{k=0}^{P-1} \langle\alpha_k^{(P)}, X_k^{(P),\mathrm{per}}\rangle$ for all $(\alpha_k)_{k\in\setZ}$ $\setC^d$-valued sequence. 
Let $(\alpha_k)_{k\in\setZ}$ a $\setC^d$-valued sequence with bounded support associated to a sequence $(\alpha_k^{(P)})_{k=0}^{P-1}$. Lemma \ref{periodic} gives
\begin{equation}\label{cvlaw}
	\lim_{P\to\infty}\mathbb{E}\left[\exp\left(\ii\mathrm{Re } X_\alpha^{(P),\mathrm{per}}\right) \right]=\exp\left(-\frac{1}{2}\sum_{(k,l)\in \setZ^2}\langle \alpha_k,C_{l-k}(\Phi^{-1})\alpha_l\rangle\right).
\end{equation}
\paragraph{Law of $n$ consecutive vertices.} From \eqref{cvlaw}, we check that any $n\in\setN^*$ consecutive vertices converge in law to a centered Gaussian random variable with block-Toeplitz covariance matrix $\Sigma^{(n-1)}=\left(C_{l-k}(\Phi^{-1})\right)_{k,l\in\set{1,\dots,n}}$. The Toeplitz property ensures that the limit marginal laws are translation invariant. 
\begin{rmk}\label{fourier_solvable}
	This convergence in law highlights the importance of lemma \ref{Invertible}. Fourier coefficients of $\Phi^{-1}$ are well defined because $\Phi$ is positive definite on the unit circle. This guarantees that we can compute many quantities of interest using periodic boundary conditions.
\end{rmk}
\paragraph{Infinite volume Gibbs measure.} 
We now give the infinite volume Gibbs measure associated to the Gaussian Markov process we constructed through thermodynamic limit. We let 
$$H=L^2(S^1\to\setC^d)$$
the Hilbert space endowed with the following norm, for any $f\in H$,
\begin{equation}
	\lVert f \rVert^2=\langle f, M(\Phi^{-1}) f \rangle_{L^2(S^1\to\setC^d)}
\end{equation}
We introduce a Gaussian process $X$ as an isometry from $H$ to a probability space:
$$X:H\to L^2(\Omega,\mathcal{F},\mathbb{P}).$$

Let $(e_i)_{i=1}^d$ be the canonical basis of $\setC^d$. This defines a translation invariant Gaussian process on $\setZ$ through
\begin{equation}\label{Gibbs_process}
	X_k=(X_{k,i})_{i=1}^d=(X(e^{\ii k\bullet}e_i))_{i=1}^d
\end{equation}
for all $k\in\setZ$, where $e^{\ii k\bullet}:\theta\mapsto e^{\ii k\theta}$ for all $\theta\in[0,2\pi]$. A direct computation shows that the process $(X_k)_{k\in\setZ}$ has the same marginal laws as the ones of the limit Gaussian Markov process computed in \eqref{cvlaw}. This gives the infinite volume Gibbs measure associated to the model. Remark that, due to infinite dimension, we are now forced to abandon densities and work with the (infinite dimensional) covariance matrix.

As for $P$ finite, the following operators are involved 
$$\begin{tikzcd}
	Q_A \arrow[r, "^{-1}"] \arrow[d, "\mathrm{FT}"]
	& C(\Phi^{-1})\arrow[l]\arrow[d,"\mathrm{FT}"] \\
	M(\Phi)\arrow[r,"^{-1}"]\arrow[u]
	& M(\Phi^{-1})\arrow[l]\arrow[u]
\end{tikzcd}$$
where "FT" stands for Fourier Transform and $Q_A=C(\Phi)$ the block-circulant operator generated by $\Phi$, see subsection \ref{section_circulant_multiplication}. Of course, as probabilists, only right part is interesting since it gives covariance matrices. Nevertheless, tridiagonality of $Q_A$ is reminiscent of our construction of a "nearest neighbor" model and, as we saw in section \ref{phi}, it leads to a second order recursive relation on Fourier coefficients of $\Phi^{-1}$. Therefore, we expect the Gaussian process $X$ to satisfy the Markov property. Indeed, this can been proved either by showing that the thermodynamic limit preserves the Markov property or by working directly on the covariance matrix of $X$. This is the purpose of the following theorem.
\begin{theorem}[Markov property]\label{markov_prop}
	The Gaussian process $(X_k)_{k\in\setZ}$ defined in \eqref{Gibbs_process} is Markov. 
\end{theorem}
\begin{proof}
	Since, for Gaussian random vectors, independence is equivalent to covariance being equal to $0$, we only have to show that, for all $k<m<l$ we have
	\begin{equation}\label{Mark}
		\covcg{X_k}{X_l}{X_m}=0.
	\end{equation}
	In other words, $X_k$ and $X_l$ are independent given $X_m$. From invariant by translation, we can restrict ourselves to the case $m=0$. From lemma \ref{Gausss}, we get
	$$\covcg{X_{k}}{X_l}{X_0}=C_{l-k}-C_{-k}C_0^{-1}C_l$$
	which is $0$ from \eqref{markovv}. 
\end{proof}
\begin{rmk}
	As a Markov process, $X$ must admit invariant boundary weights. They are determined by theorem \ref{main} and \ref{main2}, see lemma \ref{cov-inverse}.
\end{rmk}
 \begin{rmk}
 	From a Gaussian weight $\ee_{\alpha,A}$, and thus from the first-degree trigonometric polynomial $\Phi_A$, we constructed a translation invariant Gaussian Markov process on $\setZ$. Conversely, take a translation invariant Gaussian process with covariance matrix given by a multiplication operator $M(\Tilde{\Phi}^{-1})$ with some function $\Tilde{\Phi}:S^1\to H_d^+(\setC)$ continuous. If we assume this process to be Markov, we can show that $\Tilde{\Phi}$ is a trigonometric polynomial of order $1$. Indeed, Markov property \eqref{Mark} implies that Fourier coefficients of $\Tilde{\Phi}^{-1}$ follow a second order recursive relation similar to \eqref{rec} whose coefficients are Fourier coefficients of $\Tilde{\Phi}$.
 \end{rmk}

\section{Markov and Schur complement: algebraic formulas}\label{algebraicstructure}
In this section we give standard algebraic properties for Gaussian weights and Gaussian boundary densities. 
\subsection{Gaussian weights}
\begin{definition}[Gaussian weights]
	Let $d\in\setN^*$. We say that $\ee_{\alpha,A}:\setC^d \times \setC^d \to\setR_+$ is a Gaussian weight on $\setC^d$ if for all $x$, $y\in \setC^d$,
	\begin{equation}\label{gweight}
		\ee_{\alpha,A}(x,y)=\alpha \exp\left(-\frac{1}{2}\begin{pmatrix}
			x\\
			y
		\end{pmatrix}^*A\begin{pmatrix}
			x\\
			y
		\end{pmatrix}\right)   
	\end{equation}
	for $A\in H_{2d}^+(\setC)$ and $\alpha\in\setR_+^*$. We denote $\mathrm{Gauss}(\setC^d)$ the set of all Gaussian weights on $\setC^d$.
\end{definition}
We use the terminology \textit{weight} rather than \textit{transition kernel} because we consider \textit{unoriented} Markov Chains.

\begin{rmk}
	Gaussian weights correspond to densities of centered non-degenerated circularly-symmetric complex normal random variables. 
\end{rmk}

Recall that, for any $2d$-dimensional square matrix $A$ we introduce its writing with  $d$-dimensional block matrix, that is,
$$A=\begin{pmatrix}
	A_{LL} & A_{LR}\\
	A_{RL} & A_{RR}
\end{pmatrix}$$
with $A_{IJ}$ $(d\times d)$-square matrices  for all $I$ $J\in\set{L,R}$.
Gaussian weights and related partition functions come with the following product.
\begin{property}\label{schur_prod}
	We define $m:\mathrm{Gauss}(\setC^d)\times  \mathrm{Gauss}(\setC^d)\to\mathrm{Gauss}(\setC^d)$ the associative product of Gaussian weights as
	$$m(\ee_{\alpha,A},\ee_{\Tilde{\alpha},\Tilde{A}})(x,z)=\int_{\setC^d}\ee_{\alpha,A}(x,y)\ee_{\Tilde{\alpha},\Tilde{A}}(y,z)dy$$
	for $\ee_{\alpha,A}$, $\ee_{\Tilde{\alpha},\Tilde{A}}\in\mathrm{Gauss}(\setC^d)$ and for all $x$, $z\in\setC^d$. We have,
	$$m(\ee_{\alpha,A},\ee_{\beta,\Tilde{A}})=\ee_{(2\pi)^d\alpha\Tilde{\alpha}\det K, S(A,\Tilde{A})}$$
	where $K=(A_{RR}+\Tilde{A}_{LL})^{-1}$ and
	$$S(A,\Tilde{A})=\begin{pmatrix}
		A_{LL}-A_{LR}KA_{RL} & -A_{LR}K\Tilde{A}_{LR} \\
		-\Tilde{A}_{RL}KA_{RL} & \Tilde{A}_{RR}-\Tilde{A}_{RL}K\Tilde{A}_{LR}\\
	\end{pmatrix}.$$
\end{property}
\begin{rmk}
	We have $S(A,\Tilde{A})=M(A,\Tilde{A})/K$ the Schur complement  of the block $K$ of the matrix
	$$M(A,\Tilde{A})=\begin{pmatrix}
		A_{LL} & 0 & A_{LR}\\
		0 & A_{RR} & \Tilde{A}_{RL}\\
		A_{RL} & \Tilde{A}_{LR} & K
	\end{pmatrix}.$$
\end{rmk}
\begin{proof}
	Gaussian calculus. Associative property is given by Fubini's theorem. We only need to show that $S(A,\Tilde{A})\in H_{2d}^+(\setC)$. Hermitian property is immediate. Let's show positive definiteness. Let $p$, $q\in\setN^*$ and $M\in H_{p+q}^+(\setC)$ with block decomposition
	$$M=\begin{pmatrix}
		C & D\\
		D^* & E
	\end{pmatrix}.$$
	Then, for all non zero $x\in\setC^d$,
	$$x^*(M/E)x^*=(C-DE^{-1}D^*)x=\begin{pmatrix}
		x\\
		-E^{-1}D^*x
	\end{pmatrix}^*M\begin{pmatrix}
		x\\
		-E^{-1}D^*x
	\end{pmatrix} > 0.$$
	Take $M=M(A,\Tilde{A})$. It is clear that $M\in H_{3d}^+(\setC)$ since for all $x$, $y$, $z\in\setC^d$
	$$\begin{pmatrix}
		x\\
		z\\
		y
	\end{pmatrix}^*M\begin{pmatrix}
		x\\
		z\\
		y
	\end{pmatrix}=\begin{pmatrix}
		x\\
		y\\
	\end{pmatrix}^*A\begin{pmatrix}
		x\\
		y\\
	\end{pmatrix}+ \begin{pmatrix}
		y\\
		z\\
	\end{pmatrix}^*\Tilde{A}\begin{pmatrix}
		y\\
		z\\
	\end{pmatrix}>0.$$ 
	Then, it comes $S(A,\Tilde{A})=M(A,\Tilde{A})/K$ is positive definite. 
\end{proof}
\begin{rmk}
	There is no unit element for $m$ in $\mathrm{Gauss}(\setC^d)$. If we wanted one, we would need to add degenerate Gaussians to our definition.
\end{rmk}
\begin{rmk}
	The product $m$ is bilinear at the level of Gaussian weights but the product $S(A,\Tilde{A})$ is not bilinear in $A$ and $\Tilde{A}$. 
\end{rmk}
\paragraph{Schur power of a matrix.} Let $A\in H_{2d}^+(\setC)$, for any $n\in\setN^*$ we define
\begin{equation}\label{schurpower}
	A^{[n]}=S(A,S(A,\dots,S(A,A)))\in H_{2d}^+(\setC)
\end{equation}
where the product $S$ appears $n-1$ times. By associativity of $S$, this definition is unambiguous. 
\subsection{Gaussian boundary weights}
\begin{definition}[Gaussian boundary weight/density]
	Let $d\in\setN^*$. We say that $\ee_{\beta,B}':\setC^d \to\setR_+$ is a Gaussian boundary weight/density on $\setC^d$ if for all $x \in \setC^d$,
	\begin{equation}\label{gbweight}
		\ee_{\beta,B}'(x)=\beta e^{-\frac{1}{2}x^* B x}   
	\end{equation}
	
	for $B\in H_d^+(\setC)$ and $\beta\in\setR_+^*$. We denote $\mathrm{Gauss^{bc}}(\setC^d)$ the set of all Gaussian densities on $\setC^d$.
\end{definition}
\begin{rmk}
	As for Gaussian weights, a Gaussian boundary weight $\ee_{\beta,B}'$ is uniquely determined by $\beta$ and $B$ so it might simply be written $(\beta,B)$.
\end{rmk}

We also remark that Gaussian weights naturally act on Gaussian boundary densities. This is the purpose of the following property, with a proof similar to that of property  \ref{schur_prod}.

\begin{property}\label{gluingboundwithweight}
	We define $ m_L:\mathrm{Gauss}^{\mathrm{bc}}(\setC^d)\times\mathrm{Gauss}(\setC^d)\to  \mathrm{Gauss}^{\mathrm{bc}}(\setC^d)$ as
	\begin{equation*}
		m_L(\ee_{\beta,B}',\ee_{\alpha,A})(y)= \int_{\setC^d} \ee_{\beta,B}'(x)\ee_{\alpha,A}(x,y)dx
	\end{equation*}
	for all $\ee_{\beta,B}'\in\mathrm{Gauss}^{\mathrm{bc}}(\setC^d)$, $\ee_{\alpha,A}\in \mathrm{Gauss}(\setC^d)$ and all $y\in\setC^d$. We have, $$m_L(\ee_{\beta,B}',\ee_{\alpha,A})=\ee_{(2\pi)^d\alpha\beta \det(B+A_{LL})^{-1}, S_L(B,A)}$$
	where
	$$S_L(B,A)=A_{RR}-A_{RL} (B+A_{LL})^{-1} A_{LR}.$$	
	Similarly we define $m_R: \mathrm{Gauss}(\setC^d) \times \mathrm{Gauss}^{\mathrm{bc}}(\setC^d)\to  \mathrm{Gauss}^{\mathrm{bc}}(\setC^d)$ as
	\begin{equation*}
		m_R(\ee_{\alpha,A},\ee_{\beta,B}')(x)=\int_{\setC^d} \ee_{\alpha,A}(x,y)\ee_{\beta,B}'(y)dy
	\end{equation*}
	for all $\ee_{\alpha,A}\in \mathrm{Gauss}(\setC^d)$, $\ee_{\beta,B}'\in\mathrm{Gauss}^{\mathrm{bc}}(\setC^d)$ and all $x\in\setC^d$. 
	We have, $$m_R(\ee_{\beta,B}',\ee_{\alpha,A})=\ee_{(2\pi)^d\alpha\beta \det(A_{RR}+B)^{-1}, S_R(A,B)}$$
	where
	$$S_R(A,B)=A_{LL}-A_{LR} (A_{RR}+B)^{-1} A_{RL}.$$
\end{property}
We finally have the following formula with, again, a proof similar to that of property \ref{schur_prod}.
\begin{property}\label{gluingboundwithbound}
	 For all $\ee_{\beta,B}'$ and $\ee_{\Tilde{\beta},\Tilde{B}}'\in\mathrm{Gauss}^{bc}(\setC^d)$, we have
	 $$\int_{\setC^d}\ee_{\beta,B}'(x)\ee_{\Tilde{\beta},\Tilde{B}}'(x)dx=(2\pi)^d\beta\Tilde{\beta}\det(B+\Tilde{B})^{-1}\in\setR_+.$$
\end{property}
\paragraph{Notation.} Except if there is a possible confusion, we will write $m$, $m_L$ and $m_R$ with the symbol "$\cdot$". For instance,    $$m_L(\ee_{\beta,B}',\ee_{\alpha,A})=\ee_{\beta,B}'\cdot\ee_{\alpha,A}.$$

\subsection{Partition function}
We now link the Gaussian weights and their associated partition function. Let $D$ be a connected component of size $l(D)=P\in\setN^*$ and  $X^{(D)}$ a Gaussian Markov process on $D$ of weights $(\alpha_k,A_k)_{k=0}^{P-1}$, with boundary conditions $(\beta_L,B_L)$ and $(\beta_R,B_R)$. We define
$$Z_D^{\mathrm{int}}(\alpha_\bullet,A_\bullet,x_0,x_P)=\int_{\setC^{d(P-1)}}\prod_{k=0}^{P-1}\ee_{\alpha_k,A_k}(x_k,x_{k+1})dx_{1}\dots dx_{P-1}.$$

\begin{property}\label{partitionfunction}
	The function $Z_D^{\mathrm{int}}(\alpha_\bullet,A_\bullet):(x,y)\mapsto Z_D^{\mathrm{int}}(\alpha_\bullet,A_\bullet,x,y)$ is an element of $\mathrm{Gauss}(\setC^d)$ and we have
	$$Z_D^{\mathrm{int}}(\alpha_\bullet,A_\bullet)=\ee_{\alpha_0,A_0}\cdot\ee_{\alpha_2,A_2}\cdots \ee_{\alpha_{P-2},A_{P-2}}\cdot\ee_{\alpha_{P-1},A_{P-1}}$$
	and 
	$$Z_D^{\mathrm{int}}(\alpha_\bullet,A_\bullet)=m(Z_{D_1}^{\mathrm{int}}((\alpha_i)_{i=0}^{P_1-1},(A_i)_{i=0}^{P_1-1}),Z_{D_2}^{\mathrm{int}}((\alpha_i)_{i=P_1}^{P_1+P_2 -1},(A_i)_{i=P_1-1}^{P_1+P_2-1}))$$
	for all $D_1$, $D_2$ connected components of sizes $l(D_1)=P_1$ and $l(D_2)=P_2$ which form a partition of $D$. 
	Finally,
	$$Z_D(\alpha_\bullet,A_\bullet,\beta_\bullet,B_\bullet)=\ee_{\beta_L,B_L}'\ee_{\alpha_0,A_0}\cdots\ee_{\alpha_{P-1},A_{P-1}}\ee_{\beta_R,B_R}.$$
\end{property}
\begin{proof}
	Under the latter notation, direct from definition of $m$, $m_L$ and $m_R$ and associative property (implied by Fubini's theorem). 
\end{proof}
\begin{rmk}
	Remark that, for homogeneous Gaussian Markov processes of weight $(\alpha,A)$,
	\begin{equation}
		Z_D^{\mathrm{int}}(\alpha,A)=\ee_{\alpha^n,A^{[n]}}.
	\end{equation}
\end{rmk}
\section{Eigen-boundary conditions}\label{eigen-bound}
In theorem \ref{propre} we gave a fixed point condition for matrices $B_L$ and $B_R$ to define eigen-boundary conditions. In this section, we will first prove theorems \ref{main} and \ref{main2} which give an explicit recipe to recover such matrices in terms of zeros of $\Phi_A$ and their associated bases. Then, we will show that we recover the marginal laws computed in section \ref{fromtftomarkov} via periodic boundary conditions. Finally, as an application of theorems \ref{main} and \ref{main2}, we will link zeros of $\Phi_A$ with zeros of $\Phi_{A^{[n]}}$ the $n$-th Schur power of $A$, see \eqref{schurpower}.

We recall that we denote $\mathcal{S}$ the set of zeros of $\Phi_A$. The set $\mathcal{S}$ is partitioned in $\mathcal{S}_{<1}$ and $\mathcal{S}_{>1}$ the zeros inside and outside the unit disk, see lemma \ref{zeros}. Moreover, to any of these sets is associated a basis, denoted respectively $(u_w)_{w\in\mathcal{S}_{<1}}$ and $(u_w)_{w\in\mathcal{S}_{>1}}$, see definition \ref{basisindexedbyzeros}. This is, of course, under assumptions \ref{full_rank} and \ref{multiplicite}.

\subsection{Proof of Theorem \ref{main}}
This subsection is devoted to the proofs of theorem \ref{main} and of the following analogous theorem for right eigen-boundary condition.
\begin{theorem}\label{main2} Let $A\in H_{2d}^+(\setC)$ and $W_{<1}\in M_d(\setC)$ any diagonalizable matrix. Under assumptions and \ref{full_rank} and \ref{multiplicite}, the following assertions are equivalent
	\begin{itemize}
		\item[(i)] The matrix $B_R=A_{LL}+A_{LR}W_{<1}$ is right Schur-invariant for $A$.
		\item[(ii)] The matrix $W_{<1}$ is diagonal in the basis $(u_w)_{w\in\mathcal{S}_{<1}}$ indexed by zeros of $\Phi_A$,  with $W_{<1}u_w=wu_w$ for all $w\in\mathcal{S}_{<1}$.
		\item[(iii)] Let $\boldsymbol{x}\in\setC^d$. The sequence $(x_k)_{k=0}^{+\infty}$, with $x_k=W_{<1}^kx_0\in\setC^d$ for all $k\in\setN^*$, is solution to the discrete Dirichlet-type problem
		\begin{System}\label{harmonic}A_{RL}x_{k-1}+(A_{LL}+A_{RR})x_k+A_{LR}x_{k+1}=0 \text{ for all }k> 0\\
			x_0=\boldsymbol{x}\\
			\lim_{k\to +\infty}x_k=0
		\end{System}
	\end{itemize}
	Moreover, in these cases, the eigenvalue associated to $\ee'_{\beta_R,B_R}$ is 
	\begin{equation}\label{eigenval}
		\Lambda=\alpha(2\pi)^d(-1)^d\frac{\prod_{w\in\mathcal{S}_{<1}}w}{\det(A_{RL})}.
	\end{equation}
\end{theorem}
We will first show the latter theorem and, with an argument using symmetry of the problem, the proof of theorem \ref{main} will be a direct consequence. We start by giving three lemmas. 
\begin{lemma}
	Let $A\in H_{2d}^+(\setC)$ and take $w_1$ and $w_2$ in $\mathcal{S}$, we have
	\begin{equation}\label{symsym}
		\frac{1}{1-\overline{w_1}w_2}\begin{pmatrix}
			I_d\\
			w_1I_d		
		\end{pmatrix}^*A\begin{pmatrix}
			I_d\\
			w_2I_d
		\end{pmatrix}=A_{LL}+A_{LR}w_2
	\end{equation}
\end{lemma}
\begin{proof}
	Let $w_1$ and $w_2$ be in $\mathcal{S}$, we have
	\begin{align*}
		\frac{1}{1-\overline{w_1}w_2}\begin{pmatrix}
			I_d\\
			w_1I_d		
		\end{pmatrix}^*A\begin{pmatrix}
			I_d\\
			w_2I_d
		\end{pmatrix}&= \frac{1}{1-\overline{w_1}w_2}\left(  A_{LL}+A_{LR}w_{2}+A_{RL}\overline{w_1}+A_{RR}\overline{w_1}w_2\right)\\
		&=\frac{1}{1-\overline{w_1}w_2}\left(  A_{LL}(1-\overline{w_1}w_2)+A_{LR}w_{2}(1-\overline{w_1}w_2)+\Phi(w_2)\overline{w_1}w_2\right)\\
		&=A_{LL}+A_{LR}w_2. 	
	\end{align*}
\end{proof}
We recall that, in the bases indexed by zeros of $\Phi_A$, the coordinate of any $x\in\setC^d$ associated to the vector $u_w$ with $w\in\mathcal{S}_{<1}$ (resp. $w\in\mathcal{S}_{>1}$) is denoted $x(w)$, see \eqref{decomposition_dans_base_zeros}. 
\begin{lemma}\label{positivedef}
	Let $A\in H_{2d}^+(\setC)$ and assumptions \ref{full_rank} and \ref{multiplicite} fulfilled. Let $W\in M_d(\setC)$ such that $Wu_w=wu_w$ for all $w\in\mathcal{S}_{<1}$. Then, for all $x\in\setC^d$, under notation \eqref{decomposition_dans_base_zeros},
	\begin{align}
		x^*(A_{LL}+A_{LR}W)x=&\sum_{w_1,w_2\in\mathcal{S}_{<1}}\overline{x(w_1)}u_{w_1}^*\left( A_{LL} + A_{LR} w_2 \right)u_{w_2}x(w_2)\label{asym}\\
		=&	\sum_{w_1,w_2\in\mathcal{S}_{<1}}\overline{x(w_1)}\frac{1}{1-\overline{w_1}w_2}\begin{pmatrix}
			u_{w_1}\\
			w_1u_{w_1}		
		\end{pmatrix}^*A\begin{pmatrix}
			u_{w_2}\\
			w_2u_{w_2}
		\end{pmatrix}x(w_2)\label{sym}
	\end{align}
	and $(A_{LL}+A_{LR}W)$ is Hermitian and positive definite. 
\end{lemma}
\begin{proof}
	The first equality is a direct consequence of the definition of $W$ and the second one follows from \eqref{symsym}. The Hermitian property is clear from \eqref{sym}, it remains to prove that $B_R$ is positive definite. Since $\vert w_1\vert<1$ and $\vert w_2\vert<1$, 
	$$\frac{1}{1-\overline{w_1}w_2}=\sum_{k=0}^{+\infty}\overline{w_1^k}w_2^k.$$
	Therefore, \eqref{sym} leads to
	$$x^*(A_{LL}+A_{LR}W)x=\sum_{k=0}^{+\infty}\begin{pmatrix}
		W^kx\\
		W^{k+1}x
	\end{pmatrix}^*A\begin{pmatrix}
		W^kx\\
		W^{k+1}x
	\end{pmatrix}>0$$
	since $A$ is Hermitian, so every term is strictly positive for every $x\neq 0$. 
\end{proof}
\begin{lemma}\label{invtomodule}
	Let $A\in H_{2d}^+(\setC)$ and $(A_{LL}+A_{LR}W)\in H_d^+(\setC)$ with $W$ diagonalizable. For all $i\in\lbrace 1,\dots,d\rbrace$ let $u_{w_i}$ be the eigenvector of $W$ with eigenvalue $w_i\in\setC^*$. We suppose, for all $i$, $w_i\in\mathcal{S}$ and $u_{w_i}\in\ker\Phi(w_i)$. Then, for all $i\in\lbrace 1,\dots, d\rbrace$, $w_i\in\mathcal{S}_{<1}$. 
\end{lemma}
\begin{proof}
	We only have to show that every $w_i$ is of module less than $1$. The equation \eqref{symsym} and positive definiteness of $(A_{LL}+A_{LR}W)$ gives, for all $x\in\setC^d$,
	$$x^*(A_{LL}+A_{LR}W)x=	\sum_{(w_i,w_j)}\overline{x(w_i)}\frac{1}{1-\overline{w_i}w_j}\begin{pmatrix}
		u_{w_i}\\
		w_iu_{w_i}		
	\end{pmatrix}^*A\begin{pmatrix}
		u_{w_j}\\
		w_ju_{w_j}
	\end{pmatrix}x(w_j)>0.$$
	In	particular, given $i\in\lbrace 1,\dots,d\rbrace$, take $x(w_i)\neq 0$ and $x(w_j)=0$ for all $j\neq i$, it comes
	$$\vert x(w_i)\vert^2 \frac{1}{1-\vert w_i\vert^2}\begin{pmatrix}
		u_{w_i}\\
		w_iu_{w_i}		
	\end{pmatrix}^*A\begin{pmatrix}
		u_{w_i}\\
		w_iu_{w_i}
	\end{pmatrix}>0.$$
	Since $A$ is positive definite, $\vert w_i\vert<1$.
\end{proof}
We are now ready to prove theorems \ref{main2} and \ref{main}. 
\begin{proof}[Proof of Theorem \ref{main2}]
	We start by showing that $(i)\Leftrightarrow (ii)$. Suppose $(i)$ and for all $i\in\lbrace 1,\dots,d\rbrace$ let $u_{w_i}$ be the eigenvector of $W_{<1}$ with eigenvalue $w_i\in\setC^*$. Then, right Schur-invariance property \eqref{rightinv} together with the fact that $A_{LR}$ is invertible lead to
	\begin{equation}\label{phimatrix}
		A_{LL}+A_{RR}+A_{LR}W_{<1}+A_{RL}W_{<1}^{-1}=0
	\end{equation}
	which, applied to any eigenvector $u_{w_i}$ of $W_{<1}$ becomes
	$$\Phi_A(w_i)u_{w_i}=0.$$
	Therefore, $w_i\in\mathcal{S}$ and $u_{w_i}\in\ker\Phi_A(w_i)$ and lemma \ref{invtomodule} gives $w_i\in\mathcal{S}_{<1}$ so we recover $(ii)$.
	Suppose $(ii)$ and recall that, by theorem \ref{residues}, $(u_w)_{w\in\mathcal{S}_{<1}}$ is a basis of $\setC^d$. By construction of $W_{<1}$, we recover \eqref{phimatrix} which implies, since $A_{LR}$ is invertible, we get right invariance. Hermitian and positive definite properties of $B_R$ are given by lemma \ref{positivedef}. We then have $(i)$. It is clear that $(ii)\Leftrightarrow (iii)$. Note that $\lim_{k\to +\infty}x_k=0$ if and only if every eigenvalue $w$ is in $\mathcal{S}_{<1}$ (module strictly less than $1$). Finally, from theorem \ref{propre},
	$$\Lambda=\alpha(2\pi)^d\det(A_{LL}+A_{RR}+A_{LR}W_{<1})$$
	which, using \eqref{phimatrix}, gives
	\begin{equation*}
		\Lambda=\alpha(2\pi)^d\det(-A_{RL}W_{<1}^{-1})^{-1}=\alpha(2\pi)^d(-1)^d\frac{\prod_{w\in\mathcal{S}_{<1}}w}{\det(A_{RL})}.
	\end{equation*} 
\end{proof}

\begin{proof}[Proof of Theorem \ref{main}]
	Theorem \ref{main} can be deduced by remarking that if $B_L$ is a left Schur-invariant for $A$, then it is right Schur-invariant for $SAS^{-1}$. Applying theorem \ref{main2} for the matrix $SAS^{-1}$ and using identities of lemma \ref{leftright} we get the desired result. 
\end{proof}
\subsection{Comparison with Fourier approach}
Now that we gave explicit formulas for left and right eigen-boundary conditions, we are ready to fully compare this approach with the periodic boundary condition approach discussed in section \ref{fromtftomarkov}. Recall that the  link between free energies has already been given in \eqref{egalite_energies_libres}. In this section, we will first establish a correspondence between eigen-boundary conditions and Fourier coefficients of $\Phi_A^{-1}$, denoted $(C_k)_{k\in\setZ}$ (see \eqref{C_k}). Then, we will show that, under invariant boundary conditions, we do recover the marginal laws computed with periodic boundary conditions. Our construction of invariant boundaries $B_L$ and $B_R$ has an interesting application in terms of projective diagram discussed in the end of this section. 

For this whole section, fix $\alpha\in\setR_+^*$ and $A\in H_{2d}^+(\setC)$ satisfying assumptions \ref{full_rank} and \ref{multiplicite}. Let $B_L$ and $B_R$ be its associated eigen-boundary conditions, respectively given by theorems \ref{main} and \ref{main2}.
\subsubsection{Invariant boundaries and Fourier coefficients of $\Phi_A^{-1}$}
The following lemma gives an explicit link between matrices $W_{>1}$ and $W_{<1}$ and $(C_k)_{k\in\setZ}$, the Fourier coefficients of $\Phi_A^{-1}$.
\begin{lemma}\label{FCtoW}
	Let $W_{>1}$ and $W_{<1}$ respectively given by points $(ii)$ of Theorems \ref{main} and \ref{main2}. For all $k\geq 0$, we have
	\begin{equation}
		C_kC_0^{-1}=W_{>1}^{-k} \text{ and }C_{-k}C_0^{-1}=W_{<1}^k.
	\end{equation}	
\end{lemma}
\begin{proof}
	The case $k=0$ is trivial. Recall that, for all $w\in\mathcal{S}$ and all $x\in\setC^d$, $P_wx=\langle u_{1/\overline{w}}, x\rangle u_{w}$, see theorem \ref{residues}. Let $k>0$, corollary \ref{expliFC} and definition of $W_{<1}$ give 
	\begin{align*}
		W_{<1}^kC_0&=W_{<1}^k\sum_{w\in\mathcal{S}_{<1}}w^{-1}\alpha_w P_w\\
		&= \sum_{w\in\mathcal{S}_{<1}}w^{k-1}\alpha_wP_w =C_{-k}.
	\end{align*}
	Proof for the other identity is similar. 
\end{proof}
\begin{rmk}\label{markov2}
Take for instance $k$, $l\in\setN$. From the previous lemma we get
$$C_{-(k+l)}=W_{<1}^{k+l}C_0=W_{<1}^kW_{<1}^lC_0=C_{-k}C_0^{-1}C_{-l}$$
for all $k$, $l\in\setN$. More generally for all $k$, $l\in\setZ$ such that $kl\geq0$, we have
	\begin{equation}\label{markovv}
		C_{k+l}=C_kC_0^{-1}C_l.
	\end{equation}
This formula is the Markov property for the infinite volume Gaussian process $(X_k)_{k\in\setZ}$ defined in \eqref{Gibbs_process}, see theorem \ref{markov_prop}.
\end{rmk}
We deduce the following lemma which gives alternative formulas for eigen-boundary conditions $B_L$ and $B_R$.
\begin{lemma}\label{eigenbc}
	The left and right Schur-invariant matrices are of the form
	\begin{equation}
		B_{L}=A_{RR}+A_{RL}C_1C_{0}^{-1} \text { and } B_{R}=A_{LL}+A_{LR}C_{-1}C_{0}^{-1}
	\end{equation}
	Moreover,
	\begin{equation}\label{0dimweight}
		B_{L}+B_{R}=C_{0}^{-1}.
	\end{equation}
	
\end{lemma}
\begin{proof}
	Direct consequence of lemma \ref{FCtoW}. Furthermore, we have
	\begin{align*}
		B_{L}+B_{R}&=A_{LL}+A_{RR}+C_{0}^{-1}C_{-1}A_{LR}+C_{0}^{-1}C_{1}A_{RL}\\
		&=C_{0}^{-1}(C_{0}(A_{LL}+A_{RR})+C_{-1}A_{LR}+C_{1}A_{RL})\\
		&=C_{0}^{-1}
	\end{align*}
	using the recursive formula \eqref{rec}. 
\end{proof}
\subsubsection{Recovering the marginal laws computed via Fourier transform.} 
In section \ref{fromtftomarkov}, we computed marginal laws under thermodynamic limit for homogeneous Gaussian Markov processes with periodic boundary conditions, see  \eqref{cvlaw}. In particular,  the limit in law of any $P+1$ consecutive vertices is a centered Gaussian with covariance matrix given by 
\begin{equation}
	\Sigma^{(P)}=(C_{l-k})_{k,l\in\set{1,\dots,P+1}},
\end{equation}
the block-Toeplitz matrix generated by Fourier coefficients of $\Phi_A^{-1}$.
In this subsection, we show that we recover this result when taking eigen-boundary conditions $B_L$ and $B_R$. More precisely, take $X^{(P)}$ the homogeneous Gaussian Markov process of size $P$ with weights $\ee_{\alpha,A}$ and eigen-boundary conditions  $B_L$ and $B_R$. Denote $Q^{(P)}_{A,\mathrm{eigen}}\coloneqq Q^{(P)}_{(A,B_L,B_R)}$ the matrix which appears in the distribution of $X^{(P)}$, see \eqref{matrice_couplage}. The next lemma shows that $X^{(P)}$ is also a centered Gaussian with covariance matrix given by the block-Toeplitz matrix generated by Fourier coefficients of $\Phi_A^{-1}$.
\begin{lemma}\label{cov-inverse}
	Under the eigen boundary conditions $B_{L}$ and $B_{R}$, $X^{(P)}$ is a centered Gaussian vector with covariance matrix given by the block-Toeplitz matrix
	$$\Sigma^{(P)}=(C_{l-k})_{k,l\in\set{1,\dots,P+1}}.$$
\end{lemma}
\begin{proof}
	Since $X^{(P)}$ is obviously a centered Gaussian, we only have to prove that the matrix
	$Q_{A,\mathrm{eigen}}^{(P)}$ 
	is the inverse of $\Sigma^{(P)}$. The recursive formula \eqref{rec} gives 
	$$(\Sigma^{(P)} Q_{A,\mathrm{eigen}}^{(P)})_{k,l}=C_{k-l-1}A_{LR}+C_{k-l}(A_{LL}+A_{RR})+C_{k-l+1}A_{RL}=Id\delta_{k,l}$$
	for all $1<l<P+1$. Moreover, using lemma \ref{eigenbc},
	\begin{align*}
		(\Sigma^{(P)} Q_{A,\mathrm{eigen}}^{(P)})_{k,1}&=C_{1-k}(B_{L}+A_{LL})+C_{2-k}A_{RL}\\
		&=C_{1-k}(C_{0}^{-1}C_{-1}A_{LR}+A_{RR}+A_{LL})+C_{2-k}A_{RL}\\
		&=C_{-k}A_{LR}+C_{1-k}(A_{LL}+A_{RR})+C_{2-k}A_{RL}=Id\delta_{k,1}
	\end{align*}
	for all $1\leq k \leq P+1$ using equation \eqref{markovv}. Similarly, we find,
	$$(\Sigma^{(P)} Q_{A,\mathrm{eigen}}^{(P)}(A))_{k,P-1}=Id\delta_{k,P-1}.$$
\end{proof}
\paragraph{Projective diagram.} Given the covariance matrix $\Sigma^{(P)}$, we get the covariance matrix of the sub-connected component $\Sigma^{(P-1)}$ simply with a sub-block extraction, which we write $\Sigma^{(P-1)}=\mathrm{Proj}(\Sigma^{(P)})$. Given $Q_{A,\mathrm{eigen}}^{(P)}$ we could have also integrated the associated Gaussian distribution with respect to the left (or right) boundary variable to get the marginal distribution of the sub-connected component. We end up with a matrix $S(Q_{A,\mathrm{eigen}}^{(P)})$ obtained from a Schur complement of $Q_{A,\mathrm{eigen}}^{(P)}$. Thanks to eigen-boundary conditions this matrix is simply $Q_{A,\mathrm{eigen}}^{(P-1)}$, see \eqref{leftinv} and \eqref{rightinv}. We end up with the following projective diagram. 
\begin{equation}
	\begin{tikzcd}\label{eigendiagram}
		\cdots \arrow[r,"S"]& Q_{A,\mathrm{eigen}}^{(P+1)} \arrow[r, "S"] \arrow[d, "^{-1}"]& Q_{A,\mathrm{eigen}}^{(P)}\arrow[r, "S"] \arrow[d, "^{-1}"]
		& Q_{A,\mathrm{eigen}}^{(P-1)}\arrow[d,"^{-1}"]\arrow[r,"S"]& \cdots\\
		\cdots \arrow[r, "\mathrm{Proj}"]&\Sigma^{(P+1)} \arrow[r, "\mathrm{Proj}"] \arrow[u]\arrow[r, "\mathrm{Proj}"]& \Sigma^{(P)} \arrow[r, "\mathrm{Proj}"] \arrow[u]
		& \Sigma^{(P-1)}\arrow[r,"\mathrm{Proj}"]\arrow[u]& \cdots
	\end{tikzcd}.
\end{equation}
At the level of probability theory, this diagram does not have anything new. It is a simple consequence of marginalization under eigen-boundary conditions. Its translation at the level of matrices is more interesting. The family $(Q_{A,\mathrm{eigen}}^{(P)})_{P\in\setN^*}$ is a consistent family of matrices with respect to Schur "projections". The situation is different from the periodic case where we found a family of coherent but disconnected diagrams, see \eqref{periodicgdiagram}. Nevertheless, as $P\to\infty$ both approaches lead to the same operator $Q_A=C(\Phi_A)$, the block-circulant operator generated by $\Phi_A$. 
\subsection{Gluing property for $\Phi_A$}
We now give an application of theorems \ref{main} and \ref{main2}. Recall that for $n\in\setN^*$, $A^{[n]}$ is the $n$-th Schur power of $A\in H_{2d}^+(\setC)$, see \eqref{schurpower}. A consequence of theorem \ref{main2} is an easy proof of the following result which links $\Phi_A$ with $\Phi_{A^{[n]}}$. 
\begin{corollary}
	Let $A\in H_{2d}^+(\setC)$. Under assumptions \ref{full_rank} and \ref{multiplicite} and assuming  $\vert w\vert \neq \vert w' \vert$ for all $w\neq w'\in\mathcal{S}(\Phi_A)$, we have, for all $n\in\setN^*$,
	\begin{align*}
		\mathcal{S}(\Phi_A)&\to\mathcal{S}(\Phi_{A^{[n]}})\\
		w&\mapsto w^n
	\end{align*}
	is a bijection. Moreover ,$\Phi_{A^{[n]}}(w^n)u_w=0$ for all $w\in\mathcal{S}(\Phi_A)$.
\end{corollary}
\begin{proof}
	Let $A\in H_{2d}^+(\setC)$ and $n\in\setN^*$. Since $A_{LR}$ is invertible, we deduce by induction that $A_{LR}^{[n]}$ is also invertible. Therefore, by lemma \ref{zeros}, $\mathcal{S}(\Phi_{A^{[n]}})$ is composed of at most $2d$ elements. Take $W_{<1}$ as defined in theorem \ref{main2}. Then, $(x_k)_{k\in\setN}=(W_{<1}^kx_0)_{k\in\setN}$ is solution to the discrete Dirichlet type problem \eqref{harmonic}. Thanks to Gaussian elimination (Schur complement), we get in particular, for all $x_0\in\setC^d$
	\begin{equation*}A_{RL}^{[n]}x_{0}^{[n]}+(A_{LL}^{[n]}+A_{RR}^{[n]})x_{1}^{[n]}+A_{LR}^{[n]}x_{2}^{[n]}=0
	\end{equation*}
	with $x_k^{[n]}=(W_{<1}^{n})^kx_0$ for any $k\in\setN$. It implies that for any $w\in\mathcal{S}_{<1}(\Phi_A)$
	$$\Phi_{A^{[n]}}(w^n)u_w=0$$
	and $w^n\in\mathcal{S}_{<1}(\Phi_{A^{[n]}})$. A similar argument also leads to $\Phi_{A^{[n]}}(w^n)u_w=0$ for all $w\in\mathcal{S}_{>1}(\Phi_A)$. We conclude by cardinality of $\mathcal{S}(\Phi_{A^{[n]}})$. 
\end{proof}
\section{Identities in law under infinite volume Gibbs measure}\label{lawid}
For this whole section, fix $\alpha\in\setR_+^*$ and $A\in H_{2d}^+(\setC)$ satisfying assumptions \ref{full_rank} and \ref{multiplicite}. As previously discussed in remark \ref{infinite_gibbs}, eigen-boudary conditions $B_L$ and $B_R$, defined in lemma \ref{eigenbc} or, equivalently, in theorems \ref{main} and \ref{main2},  provide an infinite volume Gibbs measure under which the homogeneous Gaussian Markov process $X=(X_k)_{k\in\setZ}$ of weight $(\alpha,A)$  is well defined. This section is devoted to the computation of conditional expectations and variances of this process in terms of matrices $W_{<1}$ and $W_{>1}$ given in theorems \ref{main} and \ref{main2}. 

We start by recalling the following result concerning Gaussian random vectors. 
\begin{lemma}\label{Gausss}
	Let $d$, $d'\in\setN^*$, and $(Y,Z)$ be a centered circularly symmetric Gaussian vector with $Y\in\setC^d$ and $Z\in\setC^{d'}$ and positive definite covariance matrix
	$$\Sigma=\begin{pmatrix} \Sigma_{YY} & \Sigma_{YZ}\\
		\Sigma_{ZY} & \Sigma_{ZZ}
	\end{pmatrix}.$$
	Then,
	\begin{equation}
		\mathbb{E}[Y \vert Z]= \Sigma_{YZ}\Sigma_{ZZ}^{-1}Z \text{ and } \text{Var}[Y\vert Z]= \Sigma_{YY}-\Sigma_{YZ}\Sigma_{ZZ}^{-1}\Sigma_{ZY}.
	\end{equation}
\end{lemma}

We recall that, in the bases indexed by zeros of $\Phi_A$, the coordinate of any $x\in\setC^d$ associated to the vector $u_w$ with $w\in\mathcal{S}$ is denoted $x(w)$, see \eqref{decomposition_dans_base_zeros}. We have the following theorem.
\begin{theorem}\label{condexp} For all $l\in\setZ$, $k\in\setN$, we have
	\begin{equation}
		\mathbb{E}_{\mathrm{Gibbs}}[X_{l-k}\vert X_l]=C_{k}C_0^{-1}X_l=W_{>1}^{-k}X_l=\sum_{w\in\mathcal{S}_{>1}}w^{-k}X_l(w)u_w
	\end{equation}
	and
	\begin{equation}
		\mathbb{E}_{\mathrm{Gibbs}}[X_{l+k}\vert X_l]=C_{-k}C_0^{-1}X_l=W_{<1}^kX_l=\sum_{w\in\mathcal{S}_{<1}}w^kX_l(w)u_w.
	\end{equation} 
	Moreover, for all $k_1$, $k_2\in\setN$ such that $k_1 \leq k_2$, we have
	\begin{equation}\label{eq:condvar}
		\covcg{X_{l+k_1}}{X_{l+k_2}}{X_l}C_0^{-1}=W_{>1}^{k_1-k_2}-W_{<1}^{k_1}W_{>1}^{-k_2}.
	\end{equation}   
\end{theorem}
\begin{proof}
	We start with conditional expectations and we only prove the second line (the proof for the first one is similar). Under eigen boundary conditions, lemma \ref{cov-inverse} gives $\var{X_l}=C_0$ and $\covg{X_{l+k}}{X_l}=C_{-k}$ for all $l\in\setZ$, $k\in\setN$. Lemma \ref{Gausss} gives $\mathbb{E}_{\mathrm{Gibbs}}[X_{l-k}\vert X_{l}]=C_{-k}C_0^{-1}X_l$. Lemma \ref{FCtoW} together with definition of $W_{<1}$ gives the two other identities. Just as for conditional expectation, \eqref{eq:condvar} is a direct consequence of lemmas \ref{Gausss}, \ref{cov-inverse} and \ref{FCtoW}.
\end{proof}
\begin{rmk} Given $X_l$, the process is no longer centered and invariant by translation. We remark that the conditional expectations are given by two semi-groups generated by the matrices $W_{<1}$ and $W_{>1}$. In the bases indexed by zeros of $\Phi_A$, they are simply given by the semi-groups generated by multiplying each coordinate, respectively, by $w\in\mathcal{S}_{<1}$ and $w^{-1}$ with $w\in\mathcal{S}_{>1}$. These conditional expectations decrease to zero with exponential rate. Now, consider $\covcg{X_{l+k_1}}{X_{l+k_2}}{X_l}$ for $k_1$, $k_2\in\setN$. It is a sum of two terms. The first term, $W_{>1}^{k_1-k_2}C_0=\covg{X_{l+k_1}}{X_{l+k_2}}$ (thanks to lemmas \ref{FCtoW} and \ref{cov-inverse}), only depends on the distance between $X_{l+k_1}$ and $X_{l+k_2}$. This term is penalized with $(-W_{<1}^{k_1}W_{>1}^{-k_2}C_0)$ which depends on the distances of $X_{l+k_1}$ and $X_{l+k_2}$ from the conditioning random variable $X_l$. This penalization decreases with exponential rate as $k_1$ and $k_2$ get larger. In other words, when  $X_{l+k_1}$ and $X_{l+k_2}$ are far away from $X_l$, their conditional law is close to their law without conditioning.
\end{rmk}
\begin{rmk}
	As  we already saw, the quantities  $(W_{>1}^{-k}x_0)_{k=0}^{-\infty}$ and $(W_{<1}^kx_0)_{k=0}^{+\infty}$  solve the discrete harmonic problems \eqref{dirich_1} and \eqref{harmonic} associated to the nearest neighbor operator generated by $A$. Indeed, it is known that conditional expectations and harmonic extensions are similar objects for Gaussian Processes. Here, we recover that $(\mathbb{E}_{\mathrm{Gibbs}}[X_{l+k}\vert X_l])_{k\in\setN}$ is the harmonic extension of $X_l$. 
\end{rmk}
\section{The case $A_{LR}$ non-invertible}\label{non-invertible}
The previous sections dealt with the case $A_{LR}$ invertible, see assumption \ref{full_rank}. At the level of the function $\Phi$, studied in section \ref{phi}, this assumption is only made to ensure that $\Phi$ has $2d$ zeros counted with multiplicity, see lemma \ref{zeros} and that, with theorem \ref{residues} we have two families of $d$ vectors which form bases of $\setC^d$, see definition \ref{basisindexedbyzeros}. At the level of invariant boundaries this assumption was made to make simpler the resolution of, for instance, right Schur-invariance equation \eqref{rightinv}. When $A_{LR}$ is invertible, this equation is equivalent to \eqref{phimatrix}, see proof of theorem \ref{main2}, and this establish the link between invariance property and zeros of $\Phi$.

In this section, we show how to extend the results to the case where $A_{LR}$ is non-invertible. To do so, we will first study the function $\Phi$ in this new context. 
As we will see, it has now $2q$ zeros with $q<d$. In particular, the bases indexed by zeros of $\Phi$ now need to be completed with the bases of $\ker A_{LR}$  and $\ker A_{RL}$ to form bases of $\setC^d$. Similarly to theorem \ref{residues}, we get a decomposition of $\Phi^{-1}$ and its Fourier coefficients along this new set of vectors. Secondly, we will build left and right Schur-invariant boundary conditions in terms of matrices similar to $W_{<1}$ and $W_{>1}$ defined in theorem \ref{main} and \ref{main2}. These matrices are now diagonal in the bases indexed by zeros of $\Phi$ completed with the bases $\ker A_{LR}$ and $\ker A_{RL}$. Especially, they are no longer invertible (their kernels are respectively $\ker(A_{R
	L})$ and $\ker(A_{LR})$). 

\begin{assumptionp}{I'}\label{not_full_rank} For the whole section, let $d\in\setN^*$ and fix $A\in H_{2d}^+(\setC)$. Take $k\in \lbrace 1,\dots,d\rbrace$. We will now work under the following assumption
	\begin{equation}
		\mathrm{rank}(A_{LR})=d-k.
	\end{equation}
	Fix $(u_i^{A_{LR}})_{i=1}^k$ and $(u_i^{A_{RL}})_{i=1}^k$ respectively basis of $\ker A_{LR}$ and $\ker A_{RL}$. 
\end{assumptionp}

\subsection{A short study of $\Phi$}
The function $\Phi$ has been fully studied in section \ref{phi} in the case $A_{LR}$ invertible. Recall that any point of $\setC^*$ where $\Phi$ is not invertible is called a zero of $\Phi$, see \eqref{zeros_of_phi}. When $A_{LR}$ is not invertible, $\Phi$ has less zeros. Moreover, the bases indexed by its zeros, defined in \ref{basisindexedbyzeros}, are no longer bases of $\setC^d$. They need to be completed by the bases of $\ker A_{LR}$ and $\ker A_{RL}$. We start by giving a technical lemma which will be helpful for the whole section.   
\begin{lemma}
	Let $\Psi:\setC\to M_d(\setC)$ polynomial. Assume there exists $u\in\setC$ such that $\mathrm{rank}(\Psi(u))=d-k$ for $k\geq 2$. Then,
	\begin{equation}\label{derivee_adj}
		(\mathrm{adj}\Psi(u))^{(l)}=0
	\end{equation}
	for all $0\leq l\leq k-2$.
\end{lemma}
\begin{proof} 
	For any $n\in\setN^*$ such that $n \leq d$, we let
	$$E_n=\lbrace 1,\dots,d\rbrace \times \lbrace 1,\dots,d-1\rbrace\times\cdots \times\lbrace 1,\dots,d-n+1\rbrace.$$
	and $E_0=\emptyset$. For any $n\geq 1$, let $I_n=(i_1,\dots,i_n)\in E_n$ and $J_n=(j_1,\dots,j_n)\in E_n$. If $n=1$, for all $z\in\setC$, $\Psi_{i_1,j_1}(z)\in M_{d-1}(\setC)$ is the matrix obtained from $\Psi(z)$ by erasing its $i_1$-th row and $j_1$-th column. If $n\geq 2$, we define, recursively, for all $z\in\setC$, $\Psi_{I_n,J_n}(z)\in M_{d-n}(\setC)$ as the matrix obtained from $\Psi_{I_{n-1},J_{n-1}}(z)\in M_{d-n+1}(\setC)$ by erasing its $i_n$-th row and $j_n$-th column, with $I_{n-1}=(i_1,\dots,i_{n-1})$ and $J_{n-1}=(j_1,\dots,j_{n-1})$. If $n=0$ we keep $\Psi(z)$ by convention. 
	If we show that, for all $0\leq l<k-1$, all $1\leq n<k-l$, and any $I_{n-1}$ $J_{n-1}\in E_{n-1}$,
	\begin{align}\label{recc}
		\mathrm{adj}(\Psi_{I_{n-1},J_{n-1}}(u))^{(l)}=0
	\end{align}
	then \eqref{derivee_adj} will follow from the specific case of $n=1$, by convention. We prove \eqref{recc} recursively on $l$. Take $l=0$. Then,
	$$\mathrm{adj}\Psi_{I_{n-1},J_{n-1}}(u)=0$$
	for $z=u$ and all $n<k$ and any $I_{n-1}$, $J_{n-1}\in E_{n-1}$. Indeed, $\mathrm{adj}\Psi_{I_{n-1},J_{n-1}}(u)$ is a matrix whose coefficients are proportional to minors of order $(d-n)$ of $\Psi(u)$ which are $0$ for all $n<k$ (since $\mathrm{rank}(\Psi(u))=d-k$).  Suppose \eqref{recc} true for all $m\leq l$. Then, given $1\leq n<k-l-1$ (such $n$ exists as long as $l<k-2$) and any $I_{n-1}=(i_1,\dots,i_{n-1})\in E_{n-1}$ and $J_{n-1}=(j_1,\dots,j_{n-1})\in E_{n-1}$,
	\begin{align*}
		\left((\mathrm{adj}\Psi_{I_{n-1},J_{n-1}}(z))^{(l+1)}\right)_{i_n,j_n}&=(-1)^{i_n+j_n}(\det\Psi_{I_n,J_n}(z))^{(l+1)}\\
		&=(-1)^{i_n+j_n}\tr\left(\sum_{m=0}^l\begin{pmatrix}
			l\\
			m
		\end{pmatrix}(\mathrm{adj}\Psi_{I_n,J_n}(z))^{(m)}\Psi_{I_n,J_n}^{(l-m+1)}(z)\right).
	\end{align*}
	where $I_n=(i_1,\dots,i_{n-1},i_n)$ and $J_n=(j_1,\dots,j_{n-1},j_n)$, using Jacobi's formula and Leibniz's rule. For $z=u$ we apply $\eqref{recc}$ to $(\mathrm{adj}\Psi_{I_n,J_n}(u))^{(m)}$ for all $m\leq l$ and this conclude.
\end{proof}
\paragraph{Zeros of $\Phi$.} Recall that we denote $\mathcal{S}$ the set of zeros of $\Phi$, see \eqref{zeros_of_phi}. When $A_{RL}$ is invertible we saw that, counted with multiplicity, $\Phi$ has $2d$ zeros.  When $\mathrm{rank}(A_{RL})=d-k$ is not invertible, there are at most $2(d-k)$ zeros counted with multiplicity. This is the purpose of the following lemma. 
\begin{lemma}\label{polynome_non_inversible} Let $\Psi(z)=z\Phi(z)=A_{LR}z^2+(A_{LL}+A_{RR})z+A_{RL}$ for all $z\in\setC$. Under assumption \ref{not_full_rank}, the polynomial
	\begin{equation}\label{P}
		P(z)=\det\Psi(z)
	\end{equation}
	for all $z\in\setC$ is of degree at most $2d-k$. Moreover $0$ is of multiplicity at least $k$. Therefore, $\Phi$ admits at most $2(d-k)$ zeros counted with multiplicity. Finally, for all $w\in\mathcal{S}$, $1/\overline{w}\in\mathcal{S}$. 
\end{lemma}
\begin{proof} We write $P(z)=\sum_{l=0}^{2d}p_lz^l$ for all $z\in\setC$. Moreover, we have $\Psi(z)=z^2\Psi(1/\overline{z})^*$ for all $z\in\setC^*$ so that,
	\begin{align*}
		\det\Psi(z)=z^{2d}\det\Psi(1/\overline{z})^*= \sum_{l=0}^{2d}\overline{p_{2d-l}}z^l.
	\end{align*}
	Thus, $p_l=\overline{p_{2d-l}}$ for all $l\in\lbrace0,\dots,2d\rbrace$. Moreover, $P(0)=\det(A_{RL})=0$ by hypothesis and for all $z\in\setC$ and $l\geq 1$,
	$$P^{(l)}(z)=\tr\left(\sum_{m=0}^{l-1}\begin{pmatrix}
		l-1\\
		m
	\end{pmatrix}(\mathrm{adj}\Psi(z))^{(m)}\Psi^{(m-l+2)}(z)\right),$$
	which is zero for $z=0$ and any $l\leq k-1$ by \eqref{derivee_adj}. Then, $p_{2d-l}=\overline{p_l}=0$ for all $l\leq k-1$ and $P$ is of degree at most $2d-k$. Since for all $z\in\setC^*$, $P(z)=0$ if and only if $\det\Phi(z)=0$, we deduce that $\Phi$ admits at most $2(d-k)$ zeros. For the last assertion, see lemma \ref{zeros}.
\end{proof}
\begin{assumptionp}{II'}\label{multiplicite_non_inversible}
	For any root $x$ of $P$ defined previously in \eqref{P}, let $\mathrm{mult}(x)$ be its multiplicity. To make things easier, we make the following assumptions
	\begin{equation}
		\mathrm{mult}(w)=\dim\ker\Phi(w)=1 \text{ for all } w\in\mathcal{S}\text{, and } \mathrm{mult}(0)=k.
	\end{equation}
	In that case, $P$ is exactly of degree $2d-k$ and $\Phi$ has exactly $d-k$ non zero roots inside the unit disk and $d-k$ outside. For all $w\in\mathcal{S}$, we fix $u_w\neq 0$ such that
	\begin{equation}
		\Phi(w)=\mathrm{Vect}_\setC(u_w).
	\end{equation}
\end{assumptionp}
\paragraph{The bases indexed by zeros of $\Phi$.}
Compared to definition \ref{basisindexedbyzeros}, the families $(u_w)_{w\in\mathcal{S}_{<1}}$ and $(u_w)_{w\in\mathcal{S}_{>1}}$ are no longer bases of $\setC^d$ but, as we will see in theorem \ref{residues_non_inversible}, $((u_w)_{w\in\mathcal{S}_{<1}},(u_i^{A_{RL}})_{i=1}^k)$ and $((u_w)_{w\in\mathcal{S}_{>1}},(u_i^{A_{LR}})_{i=1}^k)$ are, where $(u_i^{A_{LR}})_{i=1}^k$ and $(u_i^{A_{RL}})_{i=1}^k$ are respectively basis of $\ker A_{LR}$ and $\ker A_{RL}$. We first need two lemmas. 
\begin{lemma}\label{adj_non_inversible} Let $\Psi(z)=z\Phi(z)=A_{LR}z^2+(A_{LL}+A_{RR})z+A_{RL}$ for all $z\in\setC$. Under assumptions \ref{not_full_rank} and \ref{multiplicite_non_inversible},
	\begin{equation}\label{adj_polynomial}
		\mathrm{adj}\Psi(z)=\sum_{l=k-1}^{2d-k-1}\Psi_lz^l
	\end{equation}
	for all $z\in\setC$ with. Moreover for all $l\in \lbrace k-1,\dots, 2d-k-1\rbrace$, $\Psi_l\in M_d(\setC)$ and
	\begin{equation}
		\Psi_l^*=\Psi_{2(d-1)-l}.
	\end{equation}
	Finally, the matrices $\Psi_{k-1}$ and $\Psi_k$ are non zero and for all $x\in\setC^d$,
	\begin{equation}\label{proj}
		\Psi_{k-1}x=\sum_{i,j=1}^{k}\alpha_{i,j}\langle u_{A_{LR}}^{(i)},x\rangle u_{A_{RL}}^{(j)}
	\end{equation}
	with $\alpha_{i,j}\in\setC$ for all $i$, $j\in\lbrace 1,\dots, k\rbrace$. 
\end{lemma}
\begin{proof}
	First of all, we remark that there exist matrices $\Psi_l \in M_d(\setC)$ for $l\in\lbrace 0,\dots,2d-2\rbrace$ such that for all $z\in\setC$,
	$$\mathrm{adj}\Psi(z)=\sum_{l=0}^{2d-2}\Psi_lz^l.$$
	Moreover, we have $\Psi(z)=z^2\Psi(1/\overline{z})^*$ for all $z\in\setC^*$ so that,
	\begin{align*}
		\mathrm{adj}\Psi(z)=z^{2(d-1)}\mathrm{adj}\Psi(1/\overline{z})^*= \sum_{l=0}^{2d-2}\Psi_{2(d-1)-l}^*z^l.
	\end{align*}
	Then, $\Psi_l^*=\Psi_{2(d-1)-l}$ for all $l\in \lbrace k-1,\dots, 2d-k-1\rbrace$. Finally, $\mathrm{rank}(\Psi(0))=\mathrm{rank}(A_{RL})=d-k$ by assumption so we can use equation \eqref{derivee_adj} to verify that $\Psi_l=0$ for all $l\in\lbrace 0,\dots, k-2\rbrace$ and therefore also for all $l\in\lbrace 2d-k,\dots, 2d-2\rbrace$ thanks to what we just showed. Finally, recall that, for all $z\in\setC$,
	$$\Psi(z)\mathrm{adj}\Psi(z)=\mathrm{adj}\Psi(z)\Psi(z)=\det(\Psi(z))I_d$$
	for $I_d$ the $d$-dimensional identity matrix. The $(k-1)$-th derivative, using equation \eqref{derivee_adj} if $k\geq 2$, together with the assumption \ref{multiplicite_non_inversible}, gives for $z=0$,
	\begin{equation}
		A_{RL}\Psi_{k-1}=\Psi_{k-1}A_{RL}=0.
	\end{equation}
	We deduce \eqref{proj}.
	
	Taking the next derivative evaluated at $z=0$ also gives
	\begin{equation}\label{adj_id}
		A_{RL}\Psi_k+(A_{LL}+A_{RR})\Psi_{k-1}=(\det\Psi(0))^{(k)}I_d.
	\end{equation}
	We have $\mathrm{rank}(A_{RL}\Psi_k)\leq \mathrm{rank}(A_{RL})=d-k$ and, since $(A_{LL}+A_{RR})$ is invertible, $\mathrm{rank}((A_{LL}+A_{RR})\Psi_{k-1})=\mathrm{rank}(\Psi_{k-1})\leq k$ thanks to \eqref{proj}. Thanks to \eqref{adj_id}, we deduce that $\Psi_{k-1}$ and $\Psi_k$ are non zero. Moreover, we get that $\mathrm{Im}(\Psi_{k-1})=\ker A_{RL}$. 
\end{proof}
\begin{rmk} Denote $T:\setC^d\to\setC^d$, $v\mapsto (A_{LL}+A_{RR})v$. We deduce from \eqref{adj_id}, that
	\begin{equation}\label{direct_sum}
		\setC^d=T(\ker A_{RL})\oplus \mathrm{Im } A_{RL}.
	\end{equation}
\end{rmk} 
The following lemma gives asymptotic formulas for $\Phi^{-1}$ which will be relevant for the residue calculus in the proof of theorem \ref{residues_non_inversible}.
\begin{lemma}
	Under assumptions \ref{not_full_rank} and \ref{multiplicite_non_inversible} and notations introduced in lemma \ref{adj_non_inversible}, we have,
	\begin{equation}\label{limite_en_0_et_infini}
		\lim_{\vert z \vert \to 0}\Phi(z)^{-1}=\frac{\Psi_{k-1}}{p_{k-1}}\text{ and }
		\lim_{\vert z \vert \to\infty}\Phi(z)^{-1}=\frac{\Psi_{2d-k-1}}{p_{2d-k}}.
	\end{equation}
\end{lemma}
\begin{proof}
	We have, for all $z\in\setC^*$, 
	$$\Phi(z)^{-1}=\frac{z\mathrm{adj}(\Psi(z))}{\det(\Psi(z))}.$$
	By lemma \ref{polynome_non_inversible}, we can write $\det\Psi(z)=z^kQ(z)$ with $Q$ polynomial such that $Q(0)=p_{k}\neq 0$. Using equation \eqref{adj_polynomial}, we get the limit in $0$. Again thanks to lemma \ref{polynome_non_inversible}, we know that $\det \Psi(z) \sim z^{2d-k}p_{2d-k}$ as $\vert z \vert \to\infty$. We get the limit at infinity with \eqref{adj_polynomial}.
\end{proof}
We are now ready to prove the following theorem, which is the extension of theorem \ref{residues} to the case where $A_{LR}$ is no longer invertible. 
\begin{theorem}\label{residues_non_inversible} For all $w\in\mathcal{S}$, $\langle u_{1/\overline{w}}, \Phi'(w) u_w \rangle\neq 0$. Let for all $w\in\mathcal{S}$, $\alpha_w\in\setC^*$ and $P_w\in M_d(\setC)$ defined as
	\begin{equation*}
		\alpha_w=\frac{1}{\langle u_{1/\overline{w}}, \Phi'(w) u_w \rangle } \text{ and } P_wx=\langle u_{1/\overline{w}},x\rangle u_w
	\end{equation*}	
	for all $x\in\setC^d$. Under assumptions \ref{not_full_rank} and \ref{multiplicite_non_inversible}, for all $z\in\setC^* \setminus \mathcal{S}$, we have
	\begin{equation}
		\Phi(z)^{-1}=\sum_{w\in\mathcal{S}}\frac{\alpha_w}{z-w}P_w + \frac{1}{p_{2d-k}}\Psi_{2d-k+1}
	\end{equation}
	Moreover, $((u_w)_{w\in\mathcal{S}_{<1}},(u_i^{A_{RL}})_{i=1}^k)$ and $((u_w)_{w\in\mathcal{S}_{>1}},(u_i^{A_{LR}})_{i=1}^k)$ are bases of $\setC^d$.
\end{theorem}
\begin{proof}
	The proof is similar to the one of theorem \ref{residues}. We let $\mathcal{C}_R$ be the circle centered at the origin of radius $R$ sufficiently large so that $\mathcal{S}\subset \mathbb{D}(0,R)$. For any $z\in\setC^*\setminus \mathcal{S}$, we have
	\begin{align*}
		\int_{\mathcal{C}_R}\Phi(u)^{-1}\frac{1}{u-z}du &=\pi\ii\left(\Phi(z)^{-1}+\sum_{w\in\mathcal{S}}\mathrm{Res}(\Phi^{-1},w)\frac{1}{w-z}\right)
	\end{align*}
	where we used limit in $0$ previously computed. Moreover, using the limit of $\Phi(u)^{-1}$ when $\vert u \vert\to\infty$ we get, 
	$$\int_{\mathcal{C}_R}\Phi(u)^{-1}\frac{1}{u-z}du=\mathrm{Res}\left(\Phi^{-1}\frac{1}{\bullet-z},\infty\right)=\frac{1}{p_{2d-k}}\Psi_{2d-k-1}.$$
	We deduce,
	$$\Phi(z)^{-1}=\sum_{w\in\mathcal{S}}\mathrm{Res}(\Phi^{-1},w)\frac{1}{z-w}+ \frac{1}{p_{2d-k}}\Psi_{2d-k-1}.$$
	To get \eqref{eq1}, we only have to show that, for all $w\in\mathcal{S}$,
	$$\mathrm{Res}(\Phi^{-1},w)=\lim_{u\to w}(u-w)\Phi(u)^{-1}=\alpha_wP_w.$$
	This has already been done in proof of theorem \ref{residues}. Finally, we recall from lemma \ref{adj_non_inversible} that for all $x\in\setC^d$,
	$$\Psi_{2d-k-1}x=\Psi_{k-1}^*x=\Psi_{k-1}x=\sum_{i,j=1}^{k}\overline{\alpha_{i,j}}\langle u_{A_{RL}}^{(i)},x\rangle u_{A_{LR}}^{(j)}.$$
	Now, just as for theorem \ref{residues}, we prove that $((u_w)_{w\in\mathcal{S}_{<1}},(u_i^{A_{RL}})_{i=1}^k)$  and $((u_w)_{w\in\mathcal{S}_{>1}},(u_i^{A_{LR}})_{i=1}^k)$ are bases of $\setC^d$. 
\end{proof}
\begin{rmk}
	As in corollary \ref{expliFC}, a direct computations shows that 
	\begin{equation}
		C_0=\sum_{w\in\mathcal{S}_{<1}}w^{-1}\alpha_wP_w+\frac{1}{p_{2d-k}}\Psi_{2d-k+1}
	\end{equation}
	and for all $k>0$,
	\begin{equation}
		C_{k}=\sum_{w\in\mathcal{S}_{>1}}-\frac{\alpha_w}{w^{k+1}} P_w\text{ and }C_{-k}=\sum_{w\in\mathcal{S}_{<1}}w^{k-1}\alpha_wP_w.
	\end{equation}
	which are now  sub-blocks matrices.
\end{rmk}

\subsection{Invariant boundaries}
We are now ready to give the following theorem which give left and right Schur-invariant boundaries. 
\begin{theorem}\label{main_not_full_rank}
	Under assumptions \ref{not_full_rank} and \ref{multiplicite_non_inversible}, let $W_{<1}\in M_d(\setC)$ such that $W_{<1}u_w=wu_w$ for all $w\in\mathcal{S}_{<1}$ and $W_{<1}u_i^{A_{RL}}=0$ for all $i\in\lbrace 1,\dots,k\rbrace$. Then, 
	\begin{equation}
		B_R=A_{LL}+A_{LR}W_{<1}\in H_d^+(\setC).
	\end{equation}
	is right Schur-invariant for $A$. Similarly, with an abuse of notation, let $W_{>1}^{-1}\in M_d(\setC)$ be the (non invertible) matrix such that $W_{>1}^{-1}u_w=(1/w)u_w$ for all $w\in\mathcal{S}_{>1}$ and $W_{>1}^{-1}u_i^{A_{LR}}=0$ for all $i\in\lbrace 1,\dots,k\rbrace$. Then,
	\begin{equation}
		B_L=A_{RR}+A_{RL}W_{>1}^{-1}
	\end{equation}
	is left Schur-invariant for $A$. 	
\end{theorem}
\begin{proof}
	Similarly to lemma \ref{positivedef}, we prove that $B_R\in H_d^+(\setC)$ and $(A_{RR} +B_R)$ is invertible. We now have to prove that $B_R$ is right Schur-invariant, that is
	$$B_R=A_{LL}-A_{LR}(A_{RR}+B_R)^{-1}A_{RL},$$
	see \eqref{rightinv}. This equality is clear for any $v\in\ker(A_{RL})$. For any $w\in\mathcal{S}_{<1}$, we have
	\begin{equation}\label{above}
		(A_{RR}+A_{LL}+A_{LR}W_{<1})u_w=-w^{-1}A_{RL}u_w
	\end{equation}
	by definition of $u_w$. Denote $f:\setC^d\to\setC^d$, $v\mapsto (A_{RR}+B_R)v$. Then, using decomposition \eqref{direct_sum} together with \eqref{above}, we get
	\begin{align*}
		f:\mathrm{Span}((u_w)_{w\in\mathcal{S}_{<1}}) \oplus \ker A_{RL}&\to \mathrm{Im}(A_{RL}) \oplus T(\ker A_{RL})\\
		\sum_{w\in\mathcal{S}_{<1}}x(w)u_w+\sum_{i=1}^kx(i)u_{i}^{A_{RL}}&\mapsto \sum_{w\in\mathcal{S}_{<1}}-w^{-1}x(w)A_{RL}u_w + \sum_{i=1}^k x(i)Tu_{i}^{A_{RL}}.
	\end{align*}
	
	Now, the restriction of $f$ on $\mathrm{Span}((u_w)_{w\in\mathcal{S}_{<1}})$, is clearly invertible with inverse given by $A_{RL}u_w\mapsto -wu_w$ for all $w\in\mathcal{S}_{<1}$. We conclude that for any $w\in\mathcal{S}_{<1}$, 
	$$(A_{LL}-A_{LR}(A_{RR}+B_R)^{-1}A_{RL})u_w=(A_{LL}+A_{LR}W_{<1})u_w=B_R$$
	and $B_R$ is right Schur-invariant. The proof is similar for $B_L$. 
\end{proof}
\begin{rmk}
	We remark that this result is similar to the one obtained under assumption \ref{full_rank}, see \ref{main} and \ref{main2}. Nevertheless, we remark that the matrices $W_{<1}$ and $W_{>1}^{-1}$ are now $0$ on $\ker A_{RL}$ and $\ker A_{LR}$ respectively. We can prove that any other result obtained in sections \ref{eigen-bound} and \ref{lawid} is still true.
\end{rmk}
\section{Application: probabilistic representation of Szegő limit theorem}\label{Szego}
In this section, after recalling the Szegő's limit theorem, we apply theorems \ref{main} and \ref{main2} to get an "eigen" version of this theorem in a simple case. In particular, we will see that there is no limit to take anymore. 
\subsection{Szegő limit theorem for block Toeplitz matrices}
We briefly recall Szegő Limit theorem. Let $\Psi:S^1\to M_d(\setC)$ be a continuous function and denote $(C_{k}(\Psi))_{k\in\setZ}$ its Fourier coefficients, see \eqref{fc}. We define the block-Toeplitz operator $T(\Psi)$ generated by $\Psi$ as the semi-infinite matrix with blocks
$$(T(\Psi))_{k,l}=C_{k-l}(\Psi)$$
for $1\leq k,l <\infty$. For all $P\geq 1$, write $T_P(\Psi)$ the truncation of $T(\Psi)$ at order $P$, that is, 
$$(T_P(\Psi))_{k,l}=C_{k-l}(\Psi)$$
for $1\leq k,l \leq P+1$.
Under some assumptions for $\Psi$, Szegő's limit theorem states
\begin{equation}\label{szeg}
	\det T_P(\Psi)\sim_{P\to\infty}\exp\left(\frac{1}{2\pi}\int_{0}^{2\pi}\log\det \Psi(e^{\ii\theta})d\theta\right)^{P+1}\kappa(\Psi).
\end{equation}
with $\kappa(\Psi)=\det T(\Psi)T(\Psi^{-1})$ for $T(\Psi)$ (resp. $T(\Psi^{-1})$) the Toeplitz operator associated to $\Psi$ (resp. $\Psi^{-1}$). 

This theorem, originally for scalar Toeplitz matrices, is due to Gabor Szegő in \cite{szego1915grenzwertsatz}. Elegant proofs for block Toeplitz using operator theory can be found in \cite{widom1974asymptotic} and, more generally for Toeplitz operator theory, we advise the reader to refer to \cite{bottcher2013analysis}.
\begin{rmk}
	Let $\Psi$ be a trigonometric polynomial of order $N\in\setN$. Take $P\geq N+1$. Up to a little modification of $T_P(\Psi)$ making it block-circulant we recover a similar asymptotic behavior of its determinant using discrete Fourier transform like for Gaussian Markov Processes. 
\end{rmk}
\begin{rmk}
	In the specific case where $\Psi$ is a trigonometric polynomial of order $1$ with values in $H_d^+(\setC)$, a direct computation of $\kappa(\Psi)$ gives,
	\begin{equation}\label{partfuncszego}
		\det T_P(\Psi)\sim_{P\to\infty}\exp\left(\frac{1}{2\pi}\int_{0}^{2\pi}\log\det \Psi(e^{\ii\theta})d\theta\right)^{P}\det(C_0(\Psi^{-1}))^{-1}.
	\end{equation}
	This expression is reminiscent of the free energy computed for periodic Gaussian Markov processes, see \eqref{free_energy_per}. We are therefore interested in an invariant version of this theorem. This is the purpose of the following subsection. 
\end{rmk}

\subsection{"Eigen" version of Szegő limit theorem for matrix valued trigonometric polynomials}
We now give an "eigen" version of Szegő limit. Inspired by Gaussian Markov processes, we introduce the following Toeplitz matrices modified at the boundary. 
\begin{definition}[Toeplitz matrix with boundary conditions]
	Let $G_L$ and $G_R\in M_d(\setC)$. The Toeplitz matrix with boundary conditions $G_L$ and $G_R$ generated by $\Psi$ is the following quasi-Toeplitz matrix
	$$(T_P^{,G_L,G_R}(\Psi))_{k,l}=\begin{cases}
		G_L+C_{0}(\Psi) \text{ for $k=l=1$}\\
		G_R+C_{0}(\Psi) \text{ for $k=l=P+1$}\\
		(T_P(\Psi))_{k,l} \text{ for $(k,l)\notin \set{(1,1), (P+1,P+1)}$}\end{cases}.$$
\end{definition}
\paragraph{Trigonometric polynomials of order $1$.} Let  $\Psi(e^{\ii\theta})=\Psi_{0}+\Psi_{1}e^{\ii\theta}+\Psi_{-1} e^{-\ii \theta}\in H_d^+(\setC)$ for all $\theta\in[0,2\pi]$, any trigonometric polynomial of order $1$. The function $\Psi$ can be treated as the function $\Phi$, in section \ref{phi}. Especially, if $\Psi$ satisfies \ref{multiplicite}, we can define, similarly theorems \ref{main} and \ref{main2}, two matrices $W_{<1}(\Psi)$ and $W_{>1}(\Psi)\in M_d(\setC)$ which are diagonal in the bases indexed by zeros of $\Psi$. Similarly to \eqref{leftinv} and \eqref{rightinv}, these matrices provide invariant boundaries for $\Psi$. For instance,
\begin{equation}\label{inv_psi}
	\Psi_1W_{<1}(\Psi)=-\Psi_1(\Psi_0+\Psi_1W_{<1}(\Psi))^{-1}\Psi_{-1}.
\end{equation}
With this good choice of boundary conditions, we get the following result. 
\begin{theorem}\label{eigen_szego}
	Let $\Psi(e^{\ii\theta})=\Psi_{0}+\Psi_{1}e^{\ii\theta}+\Psi_{-1} e^{-\ii \theta}\in H_d^+(\setC)$ for all $\theta\in[0,2\pi]$ such that $\Psi_1$ is invertible. Suppose that $\Psi$ satisfies assumptions  \ref{multiplicite}. Then, taking $G_L=\Psi_{-1}(W_{>1}(\Psi))^{-1}$ and $G_R=\Psi_1 W_{<1}(\Psi)$ for $W_{>1}(\Psi)$ and $W_{<1}(\Psi)$ as discussed above, for all $P\in\setN^*$,
	\begin{equation}\label{modified_szego}
		\det T_{P}^{G_L,G_R}=\exp\left(\frac{1}{2\pi}\int_{0}^{2\pi}\log\det \Psi(e^{\ii\theta})d\theta\right)^P\kappa(G_L,G_R)
	\end{equation}
	with $\kappa(G_L,G_R)=\det(G_L+G_R+\Psi_0)^{-1}$.
\end{theorem}
\begin{proof}
	The result follows from successive Schur complements formulas and invariant properties of $G_L$ and $G_R$, see for instance \eqref{inv_psi}. We conclude using \ref{egalite_energies_libres}.
\end{proof}
\begin{rmk} Thanks to \eqref{0dimweight}, $\kappa(G_L,G_R)=\det(C_0(\Psi^{-1}))^{-1}$ and we obviously recover \eqref{partfuncszego} computed with Szegő's limit theorem. Compared to Szegő's limit theorem (or to the block-circulant approach), there is no limit to take anymore. Nevertheless, this result only works when $\Psi\in H_d^+(\setC)$, which is the due to the Gaussian framework in which we work, see \eqref{herm_pos_def}.
\end{rmk}

\paragraph{Geometric interpretation of Szegő's limit theorem.} From a statistical physics point of view, in Szegő's limit theorem, the term $$\log\left(\exp\left(\frac{1}{2\pi}\int_{0}^{2\pi}\log\det \Psi(e^{\ii\theta})d\theta\right)/\kappa(\Psi)\right)$$ is the $0$-dimensional free energy attached to a vertex under Gibbs measure of a Gaussian Markov process on $\setZ$.

\paragraph{Trigonometric polynomials of order $N\geq2$.} When $\Psi$ is a trigonometric polynomial of order $N\geq 2$, it is still possible to modify the matrix at the boundary to obtain an "eigen" version of  Szegő's limit theorem. This is the purpose of the following result.
\begin{theorem}\label{szego_polynomial}
	Let $N\geq 2$ and $\Psi(e^{\ii\theta})=\sum_{k=-N}^N\Psi_ke^{\ii k \theta}\in H_d^+(\setC)$ for all $\theta\in[0,2\pi]$. We introduce the function $\Tilde{\Psi}(e^{\ii\theta})=\Tilde{\Psi_0}+\Tilde{\Psi}_1e^{\ii\theta} +\Tilde{\Psi}_{-1}e^{-\ii\theta}$ assumed to be in $\in H_{Nd}^+(\setC)$ with $\Tilde{\Psi_0}=T_{N-1}(\Psi)$, $\Tilde{\Psi}_1=(C_{k-l-N}(\Psi))_{k,l\in\set{0,\dots,N-1}}$ and $\Tilde{\Psi}_{-1}=\Tilde{\Psi}_{1}^*$. Suppose that  $\Tilde{\Psi}_1$ is invertible and that $\Tilde{\Psi}$ satisfies assumption \ref{multiplicite}. Then, taking $\Tilde{G}_L=\Psi_{-1}(W_{>1}(\Psi))^{-1}$ and $\Tilde{G}_R=\Psi_1 W_{<1}(\Psi)$ for $W_{>1}(\Tilde{\Psi})$ and $W_{<1}(\Tilde{\Psi})$ as discussed above, for all $P\in\setN^*$,
	\begin{equation}\label{modified_szego_2}
		\det T_{(N-1)P}^{\Tilde{G}_L,\Tilde{G}_R}=\exp\left(\frac{1}{2\pi}\int_{0}^{2\pi}\log\det \Tilde{\Psi}(e^{\ii\theta})d\theta\right)^P\Tilde{\kappa}(\Tilde{G}_L,\Tilde{G}_R)
	\end{equation}
	with $\Tilde{\kappa}(\Tilde{G}_L,\Tilde{G}_R)=\det(\Tilde{G}_L+\Tilde{G}_R+\Tilde{\Psi}_0)^{-1}$.
\end{theorem}
\begin{proof}
	The proof is similar to the one of theorem \ref{eigen_szego}.
\end{proof}
\begin{rmk}
	Since our theorems are only applicable to trigonometric  polynomials of order $2$, we had to consider the function $\Tilde{\Psi}$, which takes values in a higher dimensional vector space, rather than $\Psi$ itself. Therefore, we don't directly recover \eqref{szeg}. Nevertheless, we expect 
	$$\det T_{(N-1)P}(\Psi)\sim_{P\to\infty}\det T_{(N-1)P}^{\Tilde{G}_L,\Tilde{G}_R}(\Psi)$$
	so that
	$$\exp\left(\frac{1}{2\pi}\int_{0}^{2\pi}\log\det \Tilde{\Psi}(e^{\ii\theta})d\theta\right)=c^{1/P}\exp\left(\frac{1}{2\pi}\int_{0}^{2\pi}\log\det \Psi(e^{\ii\theta})d\theta\right)^{N-1}$$
	where
	$$c=\exp\left(\frac{1}{2\pi}\int_{0}^{2\pi}\log\det \Psi(e^{\ii\theta})d\theta\right)\frac{\kappa(\Psi)}{\kappa(\Tilde{G}_L,\Tilde{G}_R)}.$$
	We believe that this could be shown thanks to a particular study of $\Psi$ but is not the purpose of the paper. 
\end{rmk}
\bibliographystyle{plain}
\bibliography{bib}
\end{document}